\documentclass{amsart}
\usepackage[letterpaper,margin=1in]{geometry}

\usepackage{hyperref}
\usepackage{amsmath}     
\usepackage{amssymb}
\usepackage{amsthm}
\usepackage{mathtools}
\usepackage{amstext}
\usepackage{verbatim}
\usepackage{graphicx}
\usepackage{subcaption}
\usepackage{algorithm}
\usepackage{algorithmic}

\newtheorem{theorem}{Theorem}[section]

\newtheorem{proposition}[theorem]{Proposition}

\theoremstyle{definition}

\newtheorem{remark}[theorem]{Remark}

\newcommand{\rset}{\mathbb{R}}
\newcommand{\cset}{\mathbb{C}}

\newcommand{\expected}{\mathbb{E}}

\newcommand{\sgdFunction}{\beta}

\DeclareMathOperator*{\argmax}{arg\,max}

\def\Ll{\mathcal{L}}

\def\Uu{\mathcal{U}}
\def\Vv{\mathcal{V}}
\def\R{\mathbb{R}}
\def\E{\mathbb{E}}
\def\C{\mathbb{C}}
\def\dx{dx}
\def\dy{dy}
\def\DATAM{{\tilde{\rho}}}          %
\def\FTG{f}          %
\def\FTA{\beta}               %
\def\RISK{\Ll}            %
\def\ERISK{\hat\Ll}       %
\def\rmD{{\mathrm{d}}}
\def\BARS{\mathbf{S}} %
\def\ID{\mathbf{I}}
\def\IM{\mathrm{i}}
\def\BFO{\boldsymbol{\omega}}
\def\BFBH{\boldsymbol{\hat\beta}}
\def\BETAH{\hat\beta}
\def\ACTIV{s}
\def\SP#1#2{{#1}\cdot{#2}}
\def\EXPECT{\mathbb{E}}
\def\COMMA{\,,}
\def\PERIOD{\,.}
\def\THETA{\BFBH,\BFO}
\def\SUPP{{\mathrm{supp}\,}}
\def\ALPHAEX{\gamma}             %
\def\TOL{\mathrm{TOL}}  %



\thanks{
The research was supported by
Swedish Research Council grant 2019-03725. The work of P.P.\ was supported in part by the ARO Grant W911NF-19-1-0243.}

\begin{document}

\title{Adaptive random Fourier features with Metropolis sampling}

\author{Aku Kammonen}
\address{KTH Royal Institute of Technology, Stockholm, Sweden}

\author{Jonas Kiessling} 
\address{H-Ai AB, Stockholm, Sweden}

\author{Petr Plech\'{a}\v{c}}
\address{University of Delaware, Newark, DE, 19716} 

\author{Mattias Sandberg} 
\address{KTH Royal Institute of Technology, Stockholm, Sweden}

\author{Anders Szepessy}
\address{KTH Royal Institute of Technology, Stockholm, Sweden}

\maketitle

\begin{abstract}
The supervised learning problem to
determine a neural network approximation 
$\mathbb{R}^d\ni x\mapsto\sum_{k=1}^K\hat\beta_k e^{{\mathrm{i}}\omega_k\cdot x}$
with one hidden layer is studied as
a random Fourier features algorithm.  
The Fourier features, i.e., the frequencies $\omega_k\in\mathbb{R}^d$,
are sampled using an adaptive Metropolis sampler.
The Metropolis test accepts proposal frequencies $\omega_k'$, having corresponding amplitudes $\hat\beta_k'$, with the probability
$\min\big\{1, (|\hat\beta_k'|/|\hat\beta_k|)^\gamma\big\}$,
for a certain positive parameter $\gamma$, determined by minimizing the approximation error for given computational work.
This adaptive, non-parametric stochastic method leads asymptotically, as $K\to\infty$, to equidistributed amplitudes $|\hat\beta_k|$, analogous  to deterministic adaptive algorithms for differential equations. The equidistributed amplitudes are shown to asymptotically correspond to the optimal density for independent samples in random Fourier features methods.
Numerical evidence is provided in order to demonstrate the approximation properties and efficiency of the proposed algorithm. The algorithm is tested
both on synthetic data and a real-world high-dimensional benchmark.
\end{abstract}

\section{Introduction}
We consider a supervised learning problem from a data set $\{x_n,y_n\} \in \R^d\times \R$, $n=1,\dots,N$ with the data 
 independent identically distributed (i.i.d.) samples from an unknown probability distribution $\DATAM(\dx\dy)$. The distribution $\DATAM$ is not known a priori but it is accessible from samples of the data. 
We assume that there exists a function $\FTG:\R^d\to \R$ such that 
$y_n = \FTG(x_n) + \xi_n$ where the noise is represented by iid random variables 
$\xi_n$ with $\EXPECT[\xi_n]=0$ and $\EXPECT[\xi_n^2]=\sigma_\xi^2$.

We assume that the target function $\FTG(x)$ can be approximated by a single layer neural network which defines an approximation $\FTA:\R^d\times\R^{Kd}\times\C^K\to\R$ 
\begin{equation}
 \FTA(x;\BFO,\BFBH) = \sum_{k=1}^K \BETAH_k \ACTIV(\omega_k,x)\COMMA
 \end{equation}
where we use the notation for the parameters of the network $\BFBH=(\BETAH_1,\dots,\BETAH_K)\in\cset^K$, $\BFO=(\omega_1,\dots,\omega_K)\in\R^{Kd}$. We consider a particular activation function that is also known as Fourier features
\[
\ACTIV(\omega,x)= e^{{\IM}\SP{\omega}{x}}\,,\;\;\mbox{for $\omega\in \rset^d$, $x\in\rset^d$.}
\]
Here $\SP{\omega}{x} = \sum_{i=1}^d \omega^i x^i$ is the Euclidean scalar product in $\R^d$. The goal of the neural network training is to minimize, over the set of parameters $(\BFBH,\BFO)\in\C^K\times\R^{Kd}$, the risk functional
\begin{equation}\label{eq:risk}
    \RISK(\BFBH,\BFO) = \EXPECT_{\DATAM} [\ell(\FTA(x;\THETA),y)]\equiv \int \ell(\FTA(x;\THETA),y) \,\DATAM(\dx\dy)
    \PERIOD
\end{equation}
Since the distribution $\DATAM$ is not known in practice the minimization problem is solved for the empirical risk
\begin{equation}\label{eq:erisk}
    \ERISK_N(\THETA) = \frac{1}{N}\sum_{n=1}^N \ell(\FTA(x_n;\THETA),y_n)\PERIOD
\end{equation}
The network is said to be over-parametrized if the width $K$ is greater than the number $N$ of training points,
i.e., $K>N$. We shall assume a fixed width such that $N>K$ when we study the dependence on the size of training data sets.
We focus on the reconstruction with the regularized least squares type risk function 
\[
\ell(\beta(x_n;\THETA),y_n) =  |y_n - \FTA(x_n;\THETA)|^2  + \lambda \sum_{k=1}^K |\BETAH_k|^2\PERIOD
\]
The least-square functional 
is augmented by the regularization term with  a Tikhonov regularization parameter $\lambda\ge 0$. For the sake of brevity we often omit the arguments $\THETA$ and use the notation $\FTA(x)$ for $\FTA(x;\THETA)$. We also use 
$|\BFBH|^2 := \sum_{k=1}^K |\BETAH_k|^2$ for the Euclidean norm on $\C^K$.

To approximately reconstruct $\FTG$ from the data based on the least squares method is a common task in statistics and machine learning, cf. \cite{understand}, which in a basic setting
takes the form of the minimization problem

\begin{equation}\label{eq:minsquare}
\min_{\beta\in \mathcal N_K} \left\{ \E_\DATAM[|y_n-\beta(x_n)|^2]  + \lambda \sum_{k=1}^K |\hat\beta_k|^2\right\}\,,
\end{equation}
where %
\begin{equation}\label{NN}
\mathcal N_K:=\Big\{\beta( x)=\sum_{k=1}^K\hat\beta_ks(\omega_k, x) \Big\}\,,
\end{equation}
represents  an artificial neural network with one hidden layer.
Suppose we assume that the frequencies $\BFO$  are random and we denote 
$\E_{{\BFO}}[g({\BFO},x,y)]:=\E[g({\BFO},x,y)\, |\, x,y]$ the conditional expectation with respect to the distribution of ${\BFO}$ conditioned on the data $(x,y)$.
Since a minimum is always less than or equal to its mean, there holds
\begin{equation}\label{min_ett}
\min_{(\BFBH,{\BFO})\in\cset^K\times \rset^{Kd}} \left\{ \E_\DATAM[|y_n-\beta(x_n)|^2]+\lambda|\BFBH|^2\right\}
\le \E_{{\BFO}} \Big[ \min_{\BFBH\in\cset^{K}} \left\{ \E_\DATAM[|y_n-\beta(x_n)|^2]+\lambda|\BFBH|^2\right\} \Big]\,.
\end{equation}
The minimization in the right hand side of \eqref{min_ett} is also known as the {\it random Fourier features problem}, see \cite{rahimi_recht,weinan_2,rudi}.
In order to obtain a better bound in \eqref{min_ett}  we assume that $\omega_k$, $k=1,\dots,K$ are i.i.d. random variables with the common probability distribution $%
p(\omega) \rmD\omega$ and introduce a further minimization
\begin{equation}\label{E_min}
\min_{p}\E_{{\BFO}} \big[\min_{\BFBH\in\cset^{K}}\big\{ \E_\DATAM[|y_n-\beta(x_n;\THETA)|^2]+\lambda|\BFBH|^2\big\}\big]
\end{equation}
An advantage of this splitting into two minimizations is that the inner optimization is a convex problem, so that several robust solution methods are available. 
The question is: how can the density $p$ in the outer minimization be determined?

The goal of this work is to formulate a systematic method to approximately sample from an optimal distribution $p_*$. The first step is to determine the optimal distribution. Following Barron's work \cite{barron} and \cite{jones}, we first derive in Section \ref{sec_p} the known error estimate
\begin{equation}\label{barron_est}
\E_{{\BFO}}\big[ \min_{\BFBH\in\cset^K}\big\{\E_\DATAM[| \sgdFunction(x)-y|^2]+\lambda|\BFBH|^2\big\}\big]
\le  \frac{1+\lambda}{K}\E_{\omega}[\frac{|\hat f(\omega)|^2}{(2\pi)^{d}p(\omega)^2}]+ \E_\DATAM[|y-f(x)|^2]\,,
\end{equation}
based on independent samples $\omega_k$ from the distribution $p$. Then, as in importance sampling, it is shown that the right hand side is minimized by choosing $p(\omega)=p_*(\omega):= |\hat f(\omega)|/\|\hat f\|_{L^1(\rset^d)}$, where $\hat f$ is the Fourier transform of $f$.
Our next step is to formulate an adaptive method that approximately generates independent samples  from the density $p_*$, thereby following the general convergence \eqref{barron_est}.
We propose to use the Metropolis  sampler: 
\begin{itemize}
\item given frequencies ${\BFO}=(\omega_1,\ldots,\omega_k)\in\rset^{Kd}$ with corresponding amplitudes
$\BFBH=(\hat\beta_1,\ldots,\hat\beta_k)\in \cset^{K}$
a proposal ${\BFO}'\in \rset^{Kd}$
is suggested and corresponding  amplitudes
$\BFBH'\in \cset^{K}$ determined by the  minimum in \eqref{E_min}, then
\item the {\it Metropolis test} is for each $k$ to accept $\omega'_k$ with probability
$\min(1, |\hat\beta'_k|^\ALPHAEX/|\hat\beta_k|^\ALPHAEX)$.
\end{itemize}
The choice of the Metropolis criterion $\min(1, |\hat\beta'_k|^\ALPHAEX/|\hat\beta_k|^\ALPHAEX)$ and selection of $\ALPHAEX$ is explained in Remark~\ref{optimal_alpha}.
This adaptive algorithm (Algorithm~\ref{alg:ARFM}) is motivated mainly by two properties
based on the regularized empirical measure $\bar\beta(\omega):=\sum_{k=1}^K\hat\beta_k \phi_\varepsilon(\omega-\omega_k)$
related to the amplitudes $\BFBH$, 
where 
$\phi_\varepsilon(\omega)=
(2\pi\varepsilon^2)^{-d/2}e^{-|\omega|^2/(2\varepsilon^2)}$:
\begin{itemize} 
     \item[(a)] The quantities $K\,p\bar\beta$ converge to $\hat f$ in $L^1$
                     asymptotically, as $K\to\infty$ and $\varepsilon\to 0+$, as shown in
                     Proposition \ref{thm_improve_p}.
                     For the proof of Proposition \ref{thm_improve_p} we consider a simplified setting where the support of the $x$-data is all of $\mathbb{R}^d$.
\item[(b)]  Property (a) implies that the optimal density $p_*$ will asymptotically
                {\it equidistribute}  $|\bar\beta|$, i.e., $|\bar\beta|$ becomes constant since $|\hat f|/p_*=\|\hat f\|_{L^1(\rset^d)}$ is constant. 
\end{itemize}

The proposed adaptive method aims to equidistribute the 
amplitudes $|\hat\beta_k|$: if $|\hat\beta_k|$ is large more frequencies will be sampled Metropolis-wise in the neighborhood of $\omega_k$ and if $|\hat\beta_k|$ is small
then fewer frequencies will be sampled in the neighborhood. 
Algorithm \ref{alg:ARFM} includes the dramatic simplification to compute
all amplitudes in one step for the proposed frequencies, so that the computationally costly step to solve the convex minimization problem for the amplitudes is not done
for each individual Metropolis test. A reason that this simplification works is the asymptotic
independence  
$Kp|\bar\beta|\to |\hat f|$
shown in Proposition~\ref{thm_improve_p}.
We note that the regularized amplitude measure 
$\bar\beta$ is impractical
to compute in high dimension $d\gg 1$. Therefore Algorithm \ref{alg:ARFM} uses the amplitudes $\hat\beta_k$ instead
and consequently Proposition~\ref{thm_improve_p}
serves only as a motivation that the algorithm can work.

In some sense, the adaptive random features Metropolis method is a stochastic generalization of deterministic adaptive computational methods for differential equations where the optimal efficiency is obtained for equidistributed error indicators, pioneered in \cite{babuska}. In the deterministic case, additional degrees of freedom are added where the error indicators are large, e.g., by subdividing finite elements or time steps. The random features Metropolis method analogously adds frequency samples where the indicators $|\hat\beta_k|$ are large.

A common setting is to, for fixed number of data point $N$, find the  number of Fourier features $K$  with similar approximation errors as for kernel ridge regression.  
Previous such results on the kernel learning improving the sampling for random Fourier features are presented, e.g., in \cite{bach}, \cite{Wilson2013}, \cite{Li2019TowardsAU} and \cite{pmlr-v70-avron17a}. 
Our focus is somewhat different, namely 
for fixed number of Fourier features $K$ find an optimal method by adaptively adjusting the frequency sampling density for each data set.
In \cite{Wilson2013}  the Fourier features are adaptively sampled based on a density parametrized as a linear combination of Gaussians. %
The work \cite{bach} and \cite{Li2019TowardsAU} determine the optimal density as a leverage score for sampling random features, based on a singular value decomposition of an  integral operator related to the reproducing kernel Hilbert space, and formulates a method to optimally resample given samples. 
Our adaptive random feature method on the contrary is not based on a parametric description or resampling and
we are not aware of other non parametric adaptive methods generating samples for random Fourier features for general kernels.
The work \cite{pmlr-v70-avron17a} studies how to optimally choose the number of Fourier features $K$ for a given number of data points $N$ and provide upper and lower error bounds. In addition \cite{pmlr-v70-avron17a}
presents a method to effectively  sample from the leverage score in the case of Gaussian kernels.

We demonstrate computational benefits of the proposed adaptive algorithm by including a simple example that provides 
explicitly the computational complexity of the adaptive sampling Algorithm~\ref{alg:ARFM}. Numerical benchmarks in Section~\ref{sec:Benchmarks} then further document gains in
efficiency and accuracy in comparison 
with the standard random Fourier features
that use a fixed distribution of frequencies.

\medskip 

Although our analysis is carried for the specific activation
function %
$s(\omega,x)=e^{{\IM}\omega\cdot x}$,
thus directly related 
to random Fourier features approximations, we note that in the numerical experiments (see Experiment 5 in Section~\ref{sec:Benchmarks}) we also tested the activation function
\[
s(\omega,x)=\frac{1}{1 + e^{-\omega\cdot x}}\,,
\]
often used in the definition of neural networks and called the
{\it sigmoid} activation. With such a change of the activation function the concept of sampling frequencies turns into sampling weights. Numerical results in Section~\ref{sec:Benchmarks} suggest that Algorithm~\ref{alg:ARFM} performs well also
in this case. A detailed study of a more general class of activation functions is subject of ongoing work.
  \medskip 

Theoretical motivations of the algorithm are given in Sections~\ref{sec_p} and \ref{sec_amplitude}.
In Section~\ref{sec_amplitude} we formulate and prove the weak convergence of the scaled amplitudes $K\,\BFBH$.
In Section~\ref{sec_p} we derive the optimal density $p_*$ for sampling the frequencies, under the assumption that $\omega_k, k=1,\ldots, K$ are independent and $\hat f\in L^1(\rset^d)$. 
Section~\ref{sec:adaptive} describe the algorithms.
Practical consequences of the theoretical results and numerical tests with different data sets are described in 
Section~\ref{sec:Benchmarks}.
\section{Optimal frequency distribution}\label{sec_p}
\subsection{Approximation rates using a Monte Carlo method}
The purpose of this section is to derive a bound for 
\begin{equation}\label{E_omega_min_x}
    \expected_{\BFO}\big[ \min_{\BFBH\in\cset^K}\big\{\E_\DATAM[| \sgdFunction(x)-y|^2]+\lambda|\BFBH|^2\big\}\big]
\end{equation}
and apply it to estimating the approximation rate for random Fourier features.

The {\it Fourier transform}
\[
\hat f(\omega):=(2\pi)^{-d/2}\int_{\rset^d} f(x)e^{-{\IM}\omega\cdot x}\rmD x
\]
has the inverse representation
\[
f(x)=(2\pi)^{-d/2}\int_{\rset^d}\hat f(\omega)e^{{\IM}\omega\cdot x}\rmD \omega
\]
provided $f$ and $\hat f$ are $L^1(\rset^d)$ functions.
We assume  $\{\omega_1,\ldots, \omega_k\}$ are independent samples from a probability 
density $p:\rset^d\to [0,\infty)$. Then the
Monte Carlo approximation of this representation yields
the neural network approximation $f(x)\simeq \alpha(x,\BFO)$ with the estimator defined
by the empirical average
\begin{equation}\label{estimator}
\alpha(x,\BFO) = \frac{1}{K}\sum_{k=1}^K \frac{1}{(2\pi)^{d/2}}\frac{\hat f(\omega_k)}{p(\omega_k)}e^{{\IM}\omega_k\cdot x}\,.
\end{equation}
To asses the quality of this approximation we study the variance 
of the estimator $\alpha(x,\BFO)$. By construction and i.i.d. sampling of $\omega_k$ the estimator is  unbiased, that is
\begin{equation}
\E_{\BFO}[\alpha(x,\BFO)]=f(x)\,,\\
\end{equation}
and we define
\[
\hat\alpha_k :=\frac{1}{(2\pi)^{d/2}} \frac{\hat f(\omega_k)}{K\,p(\omega_k)}\,.
\]
Using this Monte Carlo approximation we obtain a bound on the error which reveals a rate of convergence with respect to the number of features $K$.
\begin{theorem}\label{lemma:bound}
Suppose the frequencies $\{\omega_1,\dots,\omega_K\}$ are i.i.d. random variables with the common distribution $p(\omega)\rmD\omega$, then 
\begin{equation}\label{variance_f}
    \mathrm{Var}_{\BFO}[\alpha(x,\BFO)] = \frac{1}{K}\E_\omega\left[\frac{|\hat f(\omega)|^2}{(2\pi)^{d}p(\omega)^2}-f^2(x)\right]\,,
\end{equation}
and
\begin{equation}\label{main_rate_merr}
\E_{{\BFO}}\big[ \min_{\BFBH\in\cset^K}\big\{\E_\DATAM[| \sgdFunction(x)-y|^2]+\lambda|\BFBH|^2\big\}\big]
\le  \frac{1+\lambda}{K}\E_{\omega}[\frac{|\hat f(\omega)|^2}{(2\pi)^{d}p(\omega)^2}]+ \E_\DATAM[|y-f(x)|^2]\,.
\end{equation}
If there is no measurement error, i.e., $\sigma_\xi^2=0$ and $y_n=f(x_n)$, then
\begin{equation}\label{main_rate}
     \E_{\BFO}\big[\min_{\BFBH\in\cset^K} \big\{\E_\DATAM[|\beta(x)-f(x)|^2]+\lambda|\BFBH|^2\big\}\big]
      \le \frac{1+\lambda}{K}\E_\omega[\frac{|\hat f(\omega)|^2}{(2\pi)^{d}p(\omega)^2}]\,.
\end{equation}
\end{theorem}

\begin{proof}
Direct calculation shows that the variance of the Monte Carlo approximation satisfies
\begin{equation}
\begin{split}
\E_{\BFO}[|\alpha(x,\BFO)-f(x)|^2]
      &=K^{-2}\E_{\BFO}\left[
         \sum_{k=1}^K\sum_{\ell=1}^K\Big(\frac{\hat f(\omega_k )e^{{\IM}\omega_k\cdot x}}{(2\pi)^{d/2}p(\omega_k)} - f(x)\Big)^*
         \Big(\frac{\hat f(\omega_\ell )e^{{\IM}\omega_\ell\cdot x}}{(2\pi)^{d/2}p(\omega_\ell)} - f(x)\Big)\right]\\
&=K^{-1} \E_\omega\left[|\frac{\hat f(\omega )e^{{\IM}\omega\cdot x}}{(2\pi)^{d/2}p(\omega)} - f(x)|^2\right]
=K^{-1}\E_\omega\left[\frac{|\hat f(\omega)|^2}{(2\pi)^{d}p(\omega)^2}-f^2(x)\right]\,,
\end{split}
\end{equation}
and since a minimum is less than or equal to its average we obtain the random feature error estimate in the case without a measurement error, i.e., $\sigma^2_\xi=0$ and $y_n=f(x_n)$,
\begin{equation*}
\begin{split}
     \E_{\BFO}\big[\min_{\BFBH\in\cset^K} \big\{\E_\DATAM[|\beta(x)-f(x)|^2]+\lambda|\BFBH|^2\big\}\big]
          &\le \E_{\BFO}\big[ \E_\DATAM[|\alpha(x)-f(x)|^2]+\lambda|\boldsymbol{\hat\alpha}|^2\big]\\
          & \le \frac{1}{K} \E[\frac{|\hat f(\omega)|^2}{(2\pi)^{d}p(\omega)^2}-f^2(x)]+\, \frac{\lambda}{K}
                  \E_\omega[\frac{|\hat f(\omega)|^2}{(2\pi)^{d}p(\omega)^2}]\\
          &\le \frac{1+\lambda}{K}\E_\omega[\frac{|\hat f(\omega)|^2}{(2\pi)^{d}p(\omega)^2}]\,.
\end{split}
\end{equation*}
Including the measurement error yields after a straightforward calculation an additional term
\[
\E_{{\BFO}}\big[ \min_{\BFBH\in\cset^K}\big\{\E_\DATAM[| \sgdFunction(x)-y|^2]+\lambda|\BFBH|^2\big\}\big]
\le  \frac{1+\lambda}{K}\E_{\omega}[\frac{|\hat f(\omega)|^2}{(2\pi)^{d}p(\omega)^2}]+ \E_\DATAM[|y-f(x)|^2]\,.
\]
\end{proof}

\subsection{Comments on the convergence rate and its complexity}\label{complexity}%
The bounds \eqref{main_rate} and \eqref{main_rate_merr} reveal the rate of convergence with respect to $K$. 
To demonstrate the computational complexity and importance of using the adaptive sampling of frequencies we fix the 
approximated function to be a simple Gaussian
$$ 
f(x)=e^{-|x|^2{\sigma}^2/2}\,,\;\;\;\mbox{with}\;\;\;
\hat f(\omega)=\frac{1}{(2\pi{\sigma}^2)^{d/2}} e^{-|\omega|^2/(2{\sigma}^2)}\,,
$$
and we consider the two cases ${\sigma}>\sqrt{2}$ and $0<{\sigma}\ll 1$. Furthermore, we choose a particular distribution $p$ by assuming the frequencies $\omega_k$, $k=1,\dots,K$ from the standard normal distribution 
$\omega_k\sim\mathcal{N}(0,1)$ (i.e., the Gaussian density $p$ with the mean zero and variance one).

\smallskip

\noindent{\it Example I (large $\sigma$)}
In the first example we assume that ${\sigma}>\sqrt{2}$, thus the integral $\int_{\rset^d}\frac{|\hat f(\omega)|^2}{p(\omega)}\rmD \omega$ is unbounded.  The error estimate \eqref{main_rate} therefore indicates no convergence. 
Algorithm~\ref{alg:ARFM} on the other hand has the optimal convergence rate for this example. 

\smallskip

\noindent{\it Example II (small $\sigma$)}
In the second example we choose $0<{\sigma}\ll 1$ thus
the convergence rate $\frac{1}{K}\int_{\rset^d}\frac{|\hat f(\omega)|^2}{p(\omega)}\rmD \omega$ in \eqref{main_rate} becomes $K^{-1}\mathcal O({\sigma}^{-d})$ while the rate is $K^{-1}\mathcal O(1)$ for the optimal distribution 
$p=p_*=|\hat f|$, as ${\sigma}\to 0+$.
The purpose of the adaptive random feature algorithm is to avoid the large factor $\mathcal O({\sigma}^{-d})$.

\medskip

To  have the loss function bounded by a given tolerance 
$\TOL$ requires therefore that the non-adaptive random feature method  uses $K\ge \TOL^{-1} \int_{\rset^d}\frac{|\hat f(\omega)|^2}{p(\omega)}\rmD \omega\simeq
\TOL^{-1}\mathcal O({\sigma}^{-d})$, and
the computational work to solve the linear least squares problem is with $N\sim K$ proportional to $K^3\simeq \TOL^{-3}\mathcal O({\sigma}^{-3d})$.

In contrast, the proposed adaptive random features Metropolis method solves the least squares problem several times with a smaller $K$ to obtain the bound $\TOL$ for the loss. The number of Metropolis steps is asymptotically determined by the diffusion approximation in \cite{roberts_1997} and becomes proportional to $\ALPHAEX d\, {\sigma}^{-2}$. 
Therefore the computational work is smaller 
$ \TOL^{-3}\mathcal O(\ALPHAEX d\, {\sigma}^{-2})$ 
for the adaptive method.

\subsection{Optimal Monte Carlo sampling.}\label{sec_3.3}
This section determines the optimal density $p$ for independent Monte Carlo samples in \eqref{variance_f}  by minimizing, with respect to $p$,
the right hand side in the variance estimate \eqref{variance_f}.
\begin{theorem}
The probability density 
\begin{equation} \label{eq:optimal_density}
p_*(\omega)= \frac{ |\widehat f(\omega)|}{\int_{\rset^d}|\widehat f(\omega')| \rmD \omega'}
\,.
\end{equation}
is the solution of the minimization problem
\begin{equation}\label{MC_loss222}
\min_{p,\int_{\R^d} p(\omega)\rmD\omega = 1}\left\{
 \frac{1}{(2\pi)^{d}}
\int_{\rset^d}\frac{|\widehat f(\omega)|^2}{p(\omega)}
\rmD \omega\right\}\,.
\end{equation}
\end{theorem}
\begin{proof}
The change of variables $p(\omega)=q(\omega)/\int_{\rset^d}q(\omega)\rmD \omega$
implies $\int_{\rset^d}p(\omega)\rmD \omega=1$ for any $q:\rset^d\to[0,\infty)$. Define for any $v:\rset^d\to \rset$ and $\varepsilon$ close to zero
\[
H(\varepsilon):=%
\int_{\rset^d}\frac{|\widehat f(\omega)|^2}{q(\omega)+\varepsilon v(\omega)}
\rmD \omega \int_{\rset^d}q(\omega)+\varepsilon v(\omega) \rmD \omega\,.
\]
At the optimum we have
\[
\begin{split}
H'(0)
= \int_{\rset^d}\frac{|\widehat f(\omega)|^2v(\omega)}{-q^2(\omega)}\rmD \omega 
\underbrace{\int_{\rset^d}q(\omega') \rmD \omega' }_{=:c_1}
 + \underbrace{\int_{\rset^d}\frac{|\widehat f(\omega')|^2}{q(\omega')} \rmD \omega'}_{=:c_2}
\int_{\rset^d}v(\omega) \rmD \omega
=
\int_{\rset^d}\big(c_2-c_1\frac{|\widehat f(\omega)|^2}{q^2(\omega)}\big)v(\omega) \rmD \omega 
\end{split}
\]
and the optimality condition 
$H'(0)=0$ implies
$q(\omega)=(\frac{c_1}{c_2})^{1/2} |\widehat f(\omega)|
$.
Consequently the optimal density becomes
\begin{equation*} %
p_*(\omega)= \frac{ |\widehat f(\omega)|}{\int_{\rset^d}|\widehat f(\omega')| \rmD \omega'}
\,.
\end{equation*}
\end{proof}
We note that the optimal density does not depend on the number of Fourier features, $K$, and the number of data points, $N$, in contrast to the optimal density for the least squares problem \eqref{eq:num_disc_prob}
derived in \cite{bach}.

As mentioned at the beginning of this section sampling $\omega_k$ from the distribution $p_*(\omega)\rmD\omega$ leads to the tight upper bound on the approximation error in \eqref{main_rate}.

\section{Asymptotic behavior of amplitudes $\hat\beta_k$}\label{sec_amplitude}
The optimal density $p_*=|\hat f|/\|\hat f\|_{L^1(\rset^d)}$ can be related to data as follows:
by considering the problem \eqref{E_omega_min_x}
and letting $\zeta(x)=\sum_{k=1}^K \hat\zeta_k e^{{\IM}\omega_k\cdot x}$ be a least squares minimizer of
\begin{equation}\label{zeta_min}
\min_{\zeta\in\mathcal{N}_K}\left\{\E_\DATAM[|\zeta(x)-y|^2\ |\ \omega] +\lambda|\hat{\boldsymbol{\zeta}}|^2\right\}
\end{equation} 
the vanishing gradient at a minimum yields the normal equations, for $\ell=1,\ldots, K$,
\[
\sum_{k=1}^K\E_\DATAM[e^{{\IM} (\omega_k-\omega_\ell)\cdot x}\hat\zeta_k] +\lambda\hat\zeta_\ell
=\E_\DATAM[y\, e^{-{\IM}\omega_\ell\cdot x}]
=\E_x[f(x)e^{-{\IM}\omega_\ell\cdot x}]\,. 
\]
Thus if the $x_n$-data points are distributed according to a distribution %
with a density $\rho:\rset^d\to [0,\infty)$ we have 
\begin{equation}\label{normaleq}
\sum_{k=1}^K\int_{\R^d} e^{{\IM} (\omega_k-\omega_\ell)\cdot x}\hat\zeta_k \rho(x) \rmD x +\lambda\hat\zeta_\ell
=\int_{\R^d} f(x)e^{-{\IM}\omega_\ell\cdot x} \rho(x)\rmD x\,
\end{equation}
and the normal equations can be written in the Fourier space as
\begin{equation}\label{limit_Z_eq}
\sum_{k=1}^K \hat\rho(\omega_\ell-\omega_k)\hat\zeta_k +\lambda\hat\zeta_\ell
=\widehat{(f\rho)}(\omega_\ell)\,,\;\;\; \ell=1,\ldots, K\,.
\end{equation}
Given the solution $\boldsymbol{\hat\zeta}_K=(\hat\zeta_1,\dots,\hat\zeta_K)$ of the normal equation \eqref{limit_Z_eq} we define $\mathbf{\hat z}_K=(\hat z_1,\dots,\hat z_K)$ 
\begin{equation}\label{zhat}
   \hat z_k := K\,p(\omega_k)\hat\zeta_k\,.
\end{equation}
Given a sequence of samples $\{\omega_k\}_{k=1}^\infty$ drawn independently from a density $p$ we impose the following assumptions:
\begin{itemize}
\item[(i)] there exists a constant $C$ such that
\begin{equation}\label{max}
     \sum_{k=1}^K |\hat\zeta_k| \equiv \frac{1}{K}\sum_{k=1}^K\frac{|\hat z_k|}{ p(\omega_k)} \le C \tag{A1}
\end{equation}
for all $K>0$,
\item[(ii)] as $K\to\infty$ we have
\begin{equation}\label{unif_b}
   \lim_{K\to\infty} \max_{k\in\{1,\ldots,K\}} |\hat\zeta_k|\equiv \lim_{K\to\infty} \max_{k\in\{1,\ldots,K\}}\frac{|\hat z_k|}{K\,p(\omega_k)}=0\,,
\tag{A2}
\end{equation}
\item[(iii)] there is a bounded open set $\mathcal{U}\subset\R^d$ such that
\begin{equation}\label{suppassume}
    \SUPP \hat f\subset \SUPP p\subset \mathcal{U}\, ,\tag{A3}
\end{equation} 
\item[(iv)] the sequence $\{\omega_k\}_{k=1}^\infty$ is dense in the support of $p$, i.e.
\begin{equation}\label{eq:dense}
\overline{\{\omega_k\}_{k=1}^\infty}=\SUPP p\, .\tag{A4}
\end{equation}
\end{itemize}
We note that \eqref{eq:dense} almost follows from \eqref{suppassume}, since that implies that the density $p$ has bounded first moment. Hence the law of large numbers implies that with probability one the sequence $\{\omega_k\}_{k=1}^\infty$ is dense in the support of $p$. 
In order to treat the limiting behaviour of $\mathbf{\hat z}_K$ as $K\to\infty$ 
we introduce the empirical measure
\begin{equation}\label{empirical_def}
\hat Z_K(\omega):=\frac{1}{K}\sum_{k=1}^K\frac{\hat z_k}{p(\omega_k)}\delta(\omega-\omega_k)\,.
\end{equation}
Thus we have for $\ell=1,\ldots, K$
\begin{equation}\label{z_conv}
\sum_{k=1}^K \hat\rho(\omega_\ell-\omega_k)\hat\zeta_k
=\int_{\rset^d}\hat\rho(\omega_\ell-\omega) \hat Z_K(\rmD \omega)
\,, 
\end{equation}
so that the normal equations \eqref{limit_Z_eq} take the form
\begin{equation}\label{normal_Z}
\int_{\rset^d}\hat\rho(\omega_\ell-\omega) \hat Z_K(\rmD \omega) +\lambda\hat\zeta_\ell=
\widehat{(f\rho)}(\omega_\ell)\,,\;\;\; \ell=1,\ldots, K\,.
\end{equation}
By the assumption \eqref{max} the empirical measures are uniformly bounded in the total variation norm
\[
\int_{\rset^d}|\hat Z_K|(\rmD \omega)= \frac{1}{K}\sum_{k=1}^K \frac{|\hat z_k|}{ p(\omega_k)}\le C\,.
\]
We note that by \eqref{suppassume} the measures $\hat Z_K$ in $\rset^d$ have their support in $\mathcal{U}$. We obtain the weak convergence result stated as the following 
Proposition.
\begin{proposition}\label{thm_improve_p}
Let $\boldsymbol{\hat\zeta}_K$ be the solution of the normal equation \eqref{limit_Z_eq} and $\hat Z_K$ the empirical measures defined by \eqref{empirical_def}.
Suppose that the assumptions \eqref{max}, \eqref{unif_b},  \eqref{suppassume}, and \eqref{eq:dense} hold, and that the density of $x$-data, $\rho$, has support on all of $\rset^d$ and satisfies $\hat\rho\in C^1$,
then 
\begin{equation}\label{thm_lim}
\lim_{\varepsilon\to 0+}\lim_{K\to\infty} 
\int_{\R^d} \phi_\varepsilon(\cdot - \omega')\,\hat Z_K(\rmD \omega')
=\hat f\,,\;\;\;\;\;\mbox{ in $L^1(\rset^d)$}\,,
\end{equation}
where $\phi_\varepsilon:\rset^d\to \rset$ are non negative smooth functions with a support in the ball $
\mathcal{B}_\varepsilon=
\{\omega\in\rset^d\,\big|\, |\omega|\le \varepsilon\}$ and
satisfying $\int_{\rset^d}\phi_\varepsilon(\omega)\rmD \omega=1$.
\end{proposition}

\begin{proof}%
To simplify the presentation we introduce 
\[
\hat\chi_{K,\varepsilon}(\omega) := \hat Z_K * \phi_\varepsilon(\omega) = \int_{\R^d} \phi_\varepsilon(\omega - \omega') \hat Z_K(\rmD \omega')\,.
\]
The proof consists of three steps:
\begin{itemize}
\item[1.] compactness yields a $L^1$ convergent subsequence  of the regularized empirical measures $\{\hat\chi_{K_j,\varepsilon}\}_{j=1}^\infty$,
\item[2.] the normal equation \eqref{limit_Z_eq} implies a related equation for the subsequence limit, and
\item[3.] a subsequence of the empirical measures converges weakly and as $\varepsilon\to 0+$ the limit normal equation establishes \eqref{thm_lim}.
\end{itemize}

{\it Step 1.} As $\phi_\varepsilon$ are standard mollifiers, we have, for a fixed $\varepsilon >0$, that the smooth functions (we omit $\varepsilon$ in the notation $\hat\chi_{K,\varepsilon}$)
\[
\hat\chi_{K} \equiv \hat Z_K*\phi_\varepsilon:\rset^d\to \cset
\] 
have uniformly bounded with respect to $K$ derivatives
$\|\nabla \hat\chi_K\|_{L^1(\rset^d)}=\mathcal O(\varepsilon^{-1})$. 
Let $\Vv$ be the Minkowski sum
$\Vv=\Uu+\mathcal{B}_\varepsilon = \{a+b : a\in\Uu, b\in\mathcal{B}_\varepsilon\}$.
By compactness, see \cite{evans}, there is a  $L^1(\Vv)$ converging subsequence of functions $\{\hat \chi_K\}$, i.e., $\hat\chi_K\to \hat\chi$ in $L^1(\Vv)$
as $K\rightarrow\infty$.
Since as a consequence of the assumption \eqref{suppassume} we have 
that $\SUPP \hat Z_K \subset \Uu$ for all $\hat Z_K$, and hence $\SUPP \hat\chi_K \subset \Vv$ for all $\hat\chi_K$, then the limit
$\hat \chi$ has its support in $\Vv$. Hence $\hat\chi$ can be extended to zero on $\rset^d\setminus \Vv$. Thus we obtain
\begin{equation}\label{g_lim}
\lim_{K\to\infty}\int_{\rset^d} g(\omega)\hat\chi_K(\rmD \omega)=\int_{\rset^d} g(\omega) \hat \chi(\omega) \rmD\omega
\end{equation}
for all $g\in C^1(\rset^d)$.

{\it Step 2.} The normal equations %
\eqref{normal_Z} can be written as a perturbation of the convergence \eqref{g_lim} using that
we have
\[
\int_{\rset^d} g(\omega)\hat Z_K(\rmD \omega)
-\int_{\rset^d} g(\omega)\hat\chi_K(\rmD \omega)
=\int_{\rset^d} \big(g(\omega)-g*\phi_\varepsilon(\omega)\big)
\hat Z_K(\rmD \omega)=\mathcal O(\varepsilon)\,.
\]
Thus we re-write the term $\int \hat\rho(\omega_\ell - \omega') \hat Z_K(\rmD\omega')$  in 
\eqref{normal_Z} as
\[
\int_{\R^d} \hat\rho(\omega-\omega') \hat Z_K(\rmD\omega') =  \int_{\R^d} \hat\rho(\omega - \omega') \hat \chi_K(\omega')\rmD\omega' + \mathcal{O}(\varepsilon)\,,
\]
now considering a general point $\omega$ instead of $\omega_l$ and the change of measure from $\hat Z_K$ to $\hat \chi_K$, 
and by Taylor's theorem%
\[
\hat\rho(\omega-\omega')=
    \hat\rho(\omega_p-\omega')
    +\hat\rho(\omega-\omega')-\hat\rho(\omega_p-\omega')
    =\hat\rho(\omega_p-\omega')
 +\int_0^1 \nabla\hat\rho\big(s\omega +(1-s)\omega_p-\omega'\big)\rmD s \cdot(\omega-\omega_p)
\]
where 
\[
\min_{p\in \{1,\ldots, K\}}|\int_0^1 \nabla\hat\rho\big(s\omega +(1-s)\omega_p-\omega'\big)\rmD s \cdot(\omega-\omega_p)|\to 0\,,\;\;\;
\mbox{as $K\to \infty$.}
\]
since by assumption the set $\{\omega_k\}_{k=1}^\infty$ is dense in the support of $p$.
Since $\lambda\hat\zeta_\ell \to 0$, as $K\to\infty$ by assumption %
\eqref{unif_b},
the normal equation %
\eqref{normal_Z} implies that
the limit is determined by
\begin{equation}\label{fp_eq}
\widehat{(f\rho)}(\omega)=\int_{\rset^d}\hat\rho(\omega-\omega') \hat \chi(\omega')\rmD \omega'+\mathcal O(\varepsilon)\,,\quad \omega\in\rset^d\,.
\end{equation}
We have here used that the function $\widehat{f\rho}$ is continuous as $\hat\rho\in C^1$, and the denseness of the sequence $\{\omega_k\}_{k=1}^\infty$.
{\it Step 3.} From the assumption \eqref{max} all $\hat Z_K$ are uniformly bounded in the total variation norm and supported on a compact set, therefore there is a weakly converging subsequence $\hat Z_K\rightharpoonup \hat Z$, i.e., for all $g\in C^1(\Vv)$
\[
   \lim_{K\to\infty}\int_{\Vv}g(\omega)\hat Z_K(\rmD \omega)\to 
   \int_{\Vv}g(\omega)\hat Z(\rmD \omega)\,.  
\]
This subsequence of $\hat Z_K$ can be chosen as a subsequence of the converging sequence $\hat\chi_K$.
Consequently we have 
\[
  \lim_{K\to\infty} \hat Z_K*\phi_\varepsilon=\hat Z*\phi_\varepsilon=\hat\chi\,.
\]
As $\varepsilon\to 0_+$ in $\hat Z*\phi_\varepsilon$ we obtain by \eqref{fp_eq}
\[
   \widehat{(f\rho)}(\omega)=\int_{\rset^d}\hat\rho(\omega-\omega') \hat Z(\rmD \omega')\,,\;\;\;\; \omega\in\rset^d\,,
\]
and we conclude, by the inverse Fourier transform, that
\[
Z(x)\rho(x)=f(x)\rho(x)\,,\;\;\;\; x\in\rset^d\,,
\]
for $f\in C_0(\rset^d)$ and $\rho$ in the Schwartz class.  
If the support of $\rho$ is $\rset^d$ we obtain that $\hat Z=\hat f\in L^1(\rset^d)$.
\end{proof}

\medskip

The approximation in Proposition ~\ref{thm_improve_p} is in the sense of the limit of the large data set, $N\to\infty$, which implies $\BFBH= \boldsymbol{\hat\zeta}_K$.
Then by the result of the proposition the regularized empirical measure for $\boldsymbol{\hat\zeta}_K$, namely $\hat Z_K*\phi_\varepsilon$, satisfies 
\[\hat Z_K*\phi_\varepsilon\underbrace{\to}_{K\to\infty} \hat f*\phi_\varepsilon\underbrace{\to}_{\varepsilon\to 0+} \hat f\,,
\mbox{ in $L^1(\rset^d)$}\,,
\]
which shows that $Kp(\omega_k)\hat\beta_k$ converges weakly to $\hat f(\omega_k)$ as $K\to\infty$ and we have $|\hat f(\omega_k)|=p_*(\omega_k) \|\hat f\|_{L^1(\rset^d)}$.
We remark that this argument gives heuristic justification for the 
proposed adaptive algorithm to work, in particular, it explains an idea behind the choice of the likelihood ratio in Metropolis accept-reject criterion. 

\begin{remark}\label{optimal_alpha}
By Proposition~\ref{thm_improve_p}  $K\, p(\omega_k)\hat\beta_k$ converges weakly 
to $\hat f(\omega_k)$ as $K\to\infty$. 
If it also converged strongly,
the asymptotic sampling density for $\omega$ in the random feature Metropolis method would  satisfy $p=(C|\hat f|)^\ALPHAEX/p^\ALPHAEX$
which  has the fixed point solution
$p=(C|\hat f|)^{\frac{\ALPHAEX}{\ALPHAEX+1}}$.
As $\ALPHAEX\to \infty$ this density $p$
approaches the optimal $|\hat f|/\|\hat f\|_{L^1(\rset^d)}$.
On the other hand the computational work increases with larger $\ALPHAEX$, in particular the number of Metropolis steps is asymptotically determined by the diffusion approximation in \cite{roberts_1997} and becomes inversely proportional to the variance of the target density, which now depends on $\ALPHAEX$. If the target density is Gaussian with the standard deviation ${\sigma}$, the number of Metropolis steps are then approximately  $\mathcal O(\ALPHAEX d {\sigma}^{-2})$ while the density is
asymptotically proportional to $|\hat f|^{\ALPHAEX/(\ALPHAEX +1)}\sim e^{-\frac{|\omega|^2\ALPHAEX}{2{\sigma}^2(\ALPHAEX +1)}}$ which yields $\int_{\rset^d}\frac{|\hat f(\omega)|^2}{p(\omega)}\rmD \omega=\mathcal O\big((1+2\ALPHAEX^{-1})^{d/2}\big)$. Thus the work for loss $\epsilon$ is roughly proportional to $\epsilon^{-3}(1+2\ALPHAEX^{-1})^{3d/2} d\ALPHAEX{\sigma}^{-2}$ which is minimal for $\ALPHAEX=3d-2$.
\end{remark}

\begin{remark}[How to remove the assumptions \eqref{max} and \eqref{unif_b}]
Assumption \eqref{max} will hold if we replace the minimization \eqref{zeta_min} by 
\begin{equation*}\label{zeta_min_2}
\min_{\zeta\in\mathcal N_K}\Big(\E[|\zeta(x)-f(x)|^2\ |\ \omega] +\lambda|\hat{\boldsymbol{\zeta}}|^2
 + \lambda_1 \sum_{k=1}^K\max(0,|\hat \zeta_k|-\frac{\lambda_2}{K})\Big)\,,
\end{equation*}
where $\lambda_1$ and $\lambda_2$ are  positive constants with
$\lambda_2> \|\frac{\hat f}{p}\|_{L^\infty(\rset^d)}$, and this additional penalty yields a least squares problem
with the same accuracy as \eqref{main_rate}, since for the optimal solution the penalty vanishes.

The other assumption \eqref{unif_b}, which is used to 
obtain $\lambda|\zeta_\ell |\to 0$ as $K\to\infty$,
can be removed by letting $\lambda$ tend to zero slowly as $K\to\infty$,
since then by \eqref{max} we obtain 
\[
\lambda|\zeta_\ell | \le \lambda \sum_{k=1}^K|\zeta_k|\le \lambda C\to 0\ \mbox{ as $K\to\infty$}.
\]
\end{remark}

\section{Description of algorithms}\label{sec:adaptive}

\begin{algorithm}[tb]
\caption{Adaptive random Fourier features with Metropolis sampling}\label{alg:ARFM}
\begin{algorithmic}
\STATE {\bfseries Input:} $\{(x_n, y_n)\}_{n=1}^N$\COMMENT{data}
\STATE {\bfseries Output:} $x\mapsto\sum_{k=1}^K\hat\beta_ke^{{\IM}\omega_k\cdot x}$\COMMENT{random features}
\STATE Choose a sampling time $T$, a proposal step length $\delta$,  an exponent $\ALPHAEX$ (see Remark \ref{optimal_alpha}), a Tikhonov parameter $\lambda$ and a frequency $m$ of  $\boldsymbol{\hat{\beta}}$ updates
\STATE $M \gets \mbox{ integer part}\, (T/\delta^2)$
\STATE ${\BFO} \gets \textit{the zero vector in $\rset^{Kd}$}$
\STATE $\boldsymbol{\hat{\beta}} \gets \textit{minimizer of the problem \eqref{eq:num_disc_prob} given } \BFO$
\FOR{$i = 1$ {\bfseries to} $M$}
    \STATE $r_{\mathcal{N}} \gets \textit{standard normal random vector in $\rset^{Kd}$}$
    \STATE $\BFO' \gets \BFO + \delta r_{\mathcal{N}}$ \COMMENT{random walk Metropolis proposal}
    \STATE $\boldsymbol{\hat{\beta}}' \gets \textit{minimizer of the problem \eqref{eq:num_disc_prob} given } \BFO'$
    \FOR{$k = 1$ {\bfseries to} $K$}
        \STATE  $r_{\mathcal{U}} \gets \textit{sample from uniform distr. on $[0,1]$}$ %
        \IF {$|\hat{\beta}'_k|^\ALPHAEX/|\hat{\beta}_k|^\ALPHAEX>r_{\mathcal{U}}$\COMMENT{Metropolis test}}
                    \STATE $\omega_{k} \gets \omega'_k$
                    \STATE $\hat{\beta}_{k} \gets \hat{\beta}'_k$
                \ENDIF
            \ENDFOR
            \IF {$i \mod m = 0$}
                \STATE $\boldsymbol{\hat{\beta}} \gets \textit{minimizer of the problem \eqref{eq:num_disc_prob} with adaptive } \BFO$
            \ENDIF
\ENDFOR
\STATE $\boldsymbol{\hat{\beta}} \gets \textit{minimizer of the problem \eqref{eq:num_disc_prob} with adaptive } \BFO$
\STATE $x\mapsto\sum_{k=1}^K\hat\beta_ke^{{\IM}\omega_k\cdot x}$

\end{algorithmic}
\end{algorithm}

In this section we formulate the adaptive random features Algorithm 1, and its extension, Algorithm 2, which adaptively updates the covariance matrix when sampling frequencies $\BFO$. Both algorithms are tested on different data sets and the tests are described in Section \ref{sec:Benchmarks}.

Before running Algorithm \ref{alg:ARFM} or \ref{alg:ARFM_acov} we normalize all training data to have mean zero and component wise standard deviation one. The normalization procedure is described in Algorithm \ref{alg:normalization_of_data}.

A discrete version of problem \eqref{eq:minsquare} can be formulated, for training data $\{(x_n, y_n)\}_{n=1}^N$, as the standard least squares problem
\begin{equation}\label{eq:num_disc_prob}
    \min_{\boldsymbol{\hat{\beta}}\in \mathbb{C}^K}\left\{N^{-1}|\BARS\boldsymbol{\hat{\beta}}-\mathbf y|^2 + \lambda|\boldsymbol{\hat{\beta}}|^2\right\}
\end{equation}
where $\BARS\in\mathbb{C}^{N\times K}$ is the matrix with elements $\BARS_{n,k} = e^{{\IM}\omega_k\cdot x_n}$, $n = 1,...,N$, $k = 1,...,K$ and $\mathbf y=(y_1,\ldots,y_N)\in \mathbb{R}^N$. Problem \eqref{eq:num_disc_prob} has the corresponding linear normal equations
\begin{equation}\label{eq:num_normal_eq}
    (\BARS^T\BARS+\lambda N \ID)\boldsymbol{\hat{\beta}} = \BARS^T\mathbf y
\end{equation}
which can be solved e.g. by singular value decomposition if $N\sim K$ or by the stochastic gradient method if $N\gg K$,  cf. \cite{trefethen_bau_1997} and \cite{understand}. Other alternatives when $N\gg K$ are to sample the data points
using information about the kernel in the random features, see e.g.\,\cite{bach}.
Here we do not focus on the interesting and important question of how to optimally sample data points. Instead we focus on how to sample random features, i.e., the frequencies $\omega_k$.

In the random walk Metropolis proposal step in Algorithm \ref{alg:ARFM} the vector $r_{\mathcal{N}}$ is a sample from the standard multivariate normal distribution. The choice of distribution to sample $r_{\mathcal{N}}$ from is somewhat arbitrary and not always optimal. Consider for example a target distribution in two dimensions that has elliptically shaped level surfaces. From a computational complexity point of view one would like to take different step lengths in different directions.

The computational complexity reasoning leads us to consider to sample $r_{\mathcal{N}}$ from a multivariate normal distribution with a covariance matrix $C_t$ adaptively updated during the $M$ iterations. The general idea of adaptively updating a covariance matrix during Metropolis iterations is not novel. A recursive algorithm for adaptively updating the covariance matrix is proposed and analysed in \cite{haario2001} and further test results are presented in \cite{roberts_examples}.

In Algorithm~\ref{alg:ARFM} we choose the initial step length value $\delta = 2.4^2/d$ which is motivated for general Metropolis sampling in \cite{roberts2001}. When running Algorithm~\ref{alg:ARFM} the value of $\delta$ can be adjusted as a hyperparameter depending on the data.

In Algorithm~\ref{alg:ARFM_acov} there is the hyperparameter $\omega_{\text{max}}$ which defines the maximum radius of frequencies $\omega$ that can be sampled by Algorithm~\ref{alg:ARFM_acov}. In some problems the sampling of frequencies will start to diverge unless $\omega_{\text{max}}$ is finite. This will typically happen if the frequency distribution is slowly decaying as $|\omega|\to\infty$. %
In the convergence proof of \cite{haario2001} the probability density function of the distribution to be sampled from is required to have compact support. In practice though when applying Algorithm~\ref{alg:ARFM_acov} we notice that a normal distribution approximates compactness sufficiently well to not require $\omega_{\text{max}}$ to be finite. Another approach is to let $\omega_{\text{max}}$ be infinity and adjust the hyperparamerters $\delta$ and $M$ so that the iterations does not diverge from the minima.
All hyperparameters are adjusted to minimize the error computed on a validation set disjoint from the training and test data sets.

With adaptive covariance combined into Algorithm \ref{alg:ARFM} we sample $r_{\mathcal{N}}$ from $ \mathcal{N}(\boldsymbol{0}, \Bar{C})$ where initially $\Bar{C}$ is the identity matrix in $\mathbb{R}^{d\times d}$. After each iteration $i = 1,2,...,M$ 
the covariance $\Bar{C}'$ of all previous frequencies $\omega_k^j$, $k = 1,2,...,K$, $j < i$ is computed. After $t_0$ iterations we update $\Bar{C}$ with the value of $\Bar{C}'$.
We present adaptive covariance applied to Algorithm \ref{alg:ARFM} in Algorithm \ref{alg:ARFM_acov}.

The initial value of $\BFO$ is set to the zero vector in $\rset^{Kd}$ so at iteration $i = 1$, the proposal frequency $\BFO'$ will be a sample from $\mathcal{N}(\boldsymbol{0}, \text{diag}([\delta^2, \delta^2,..., \delta^2])$ where $[\delta^2, \delta^2,..., \delta^2]$ is a vector in $\mathbb{R}^d$.

\begin{algorithm}[tb]
\caption{Adaptive random Fourier features with Metropolis sampling and adaptive covariance}\label{alg:ARFM_acov}
\begin{algorithmic}
\STATE {\bfseries Input:} $\{(x_n, y_n)\}_{n=1}^N$\COMMENT{data}
\STATE {\bfseries Output:} $x\mapsto\sum_{k=1}^K\hat\beta_ke^{{\IM}\omega_k\cdot x}$\COMMENT{random features}
\STATE Choose a sampling time $T$, a proposal step length $\delta$,  an exponent $\ALPHAEX$ (see Remark \ref{optimal_alpha}), a Tikhonov parameter $\lambda$, a burn in time $t_0$ for the adaptive covariance, a maximum frequency radius $\omega_{\text{max}}$ and a number $\check{N}$ of $\boldsymbol{\hat{\beta}}$ updates
\STATE $M \gets \mbox{ integer part}\, (T/\delta^2)$
\STATE ${\BFO} \gets \textit{the zero vector in $\rset^{Kd}$}$
\STATE $\boldsymbol{\hat{\beta}} \gets \textit{minimizer of the problem \eqref{eq:num_disc_prob} given } \BFO$
\STATE $S_{\omega} \gets 0$
\STATE $C_{\omega} \gets \textit{the zero matrix in } \mathbb{R}^{d\times d}$
\STATE $\Bar{C} \gets \textit{identity matrix in }\mathbb{R}^{d\times d}$
\FOR{$i = 1$ {\bfseries to} $M$}
    \STATE $r_{\mathcal{N}} \gets \textit{sample from $\mathcal{N}(\boldsymbol{0}, \Bar{C})$}$
    \STATE $\BFO' \gets \BFO + \delta r_{\mathcal{N}}$ \COMMENT{random walk Metropolis proposal}
    \STATE $\boldsymbol{\hat{\beta}}' \gets \textit{minimizer of the problem \eqref{eq:num_disc_prob} given } \BFO'$
    \FOR{$k = 1$ {\bfseries to} $K$}
        \STATE  $r_{\mathcal{U}} \gets \textit{sample from uniform distr. on $[0,1]$}$ %
        \IF {$|\hat{\beta}'_k|^\ALPHAEX/|\hat{\beta}_k|^\ALPHAEX>r_{\mathcal{U}}$ \AND $|\omega_k'| < \omega_{\text{max}}$\COMMENT{Metropolis test}}
                    \STATE $\omega_{k} \gets \omega'_k$
                    \STATE $\hat{\beta}_{k} \gets \hat{\beta}'_k$
                \ENDIF
                \STATE $S_{\omega} \gets S_{\omega} + \omega_k$
                \STATE $S_{C} \gets S_{C} + \omega_k^T\omega_k$
            \ENDFOR
            \STATE $\Bar{\omega'} \gets S_{\omega}/(iK)$
            \STATE $\Bar{C'} \gets S_{C}/(iK) - \Bar{\omega'}^T\Bar{\omega'}$
            \IF {$i > t_0$}
            \STATE $\Bar{C} \gets \Bar{C'}$
            \ENDIF
            \IF {$i \mod m = 0$}
                \STATE $\boldsymbol{\hat{\beta}} \gets \textit{minimizer of the problem \eqref{eq:num_disc_prob} with adaptive } \BFO$
            \ENDIF
\ENDFOR
\STATE $\boldsymbol{\hat{\beta}} \gets \textit{minimizer of the problem \eqref{eq:num_disc_prob} with adaptive } \BFO$
\STATE $x\mapsto\sum_{k=1}^K\hat\beta_ke^{{\IM}\omega_k\cdot x}$

\end{algorithmic}
\end{algorithm}

\begin{algorithm}[tb]
\caption{Normalization of data}\label{alg:normalization_of_data}
\begin{algorithmic}
\STATE {\bfseries Input:} $\{(x_n, y_n)\}_{n=1}^N$\COMMENT{data}
\STATE {\bfseries Output:} $\{(x_n, y_n)\}_{n=1}^N$\COMMENT{normalized data}
\STATE $\bar y \gets \frac{1}{N}\sum_{n=1}^N y_n$
\STATE $\bar x^j\gets \frac{1}{N}\sum_{n=1}^N x_n^j,\, j = 1,2,...,d$
\STATE $\sigma_y \gets \sqrt{\frac{\sum_{n=1}^N(y_n-\bar y)^2}{N-1}}$
\STATE $\sigma_{x^j} \gets \sqrt{\frac{\sum_{n=1}^N(x_n^j-\bar x^j)^2}{N-1}},\, j = 1,2,...,d$
\STATE $\{(x_n, y_n)\}_{n=1}^N \gets \{(\frac{x_n^1-\bar{x}_n^1}{\sigma_{x^1}}, \frac{x_n^2-\bar{x}_n^2}{\sigma_{x^2}},...,\frac{x_n^d-\bar{x}_n^d}{\sigma_{x^d}}; \frac{y_n - \bar y}{\sigma_y})\}_{n=1}^N$
\end{algorithmic}
\end{algorithm}

\section{Numerical tests}\label{sec:Benchmarks}
We demonstrate different capabilities of the proposed algorithms with three numerical case studies. The first two cases are regression problems and the third case is a classification problem. For the regression problems we also show comparisons with the stochastic gradient method.

The motivation for the first case is to compare the results of the algorithms to the estimate \eqref{main_rate} based on the constant $\E_\omega[\frac{|\hat f(\omega)|^2}{(2\pi)^{d}p^2(\omega)}]$ which is minimized for $p= |\hat f|/\|\hat f\|_{L^1(\rset^d)}$. Both Algorithm~\ref{alg:ARFM} and Algorithm~\ref{alg:ARFM_acov} approximately sample the optimal distribution but for example a standard random Fourier features approach with $p \sim \mathcal{N}(0, 1)$ does not.

Another benefit of Algorithm~\ref{alg:ARFM} and especially Algorithm~\ref{alg:ARFM_acov} is the efficiency in the sense of computational complexity. The purpose of the second case is to study the development of the generalization error over actual time in comparison with a standard method which in this case is an implementation of the stochastic gradient method. The problem is in two dimensions which imposes more difficulty in finding the optimal distribution compared to a problem in one dimension.

In addition to the regression problems in the first two cases we present a classification problem in the third case. It is the classification problem of handwritten digits, with labels, found in the MNIST database. For training the neural network we use Algorithm~\ref{alg:ARFM} and compare with using naive random Fourier features. The purpose of the third case is to demonstrate the ability of Algorithm~\ref{alg:ARFM} to handle non synthetic data.
In the simulations for the first case we perform five experiments:
\begin{itemize}
    \item Experiment 1: The distribution of the frequencies $\BFO\in\rset^{Kd}$ is obtained adaptively by Algorithm~\ref{alg:ARFM}.
    \item Experiment 2: The distribution of the frequencies $\BFO\in\rset^{Kd}$ is obtained adaptively by Algorithm~\ref{alg:ARFM_acov}.
    \item Experiment 3: The distribution of the frequencies $\BFO\in\rset^{Kd}$ is fixed and the independent components $\omega_k$ are sampled from a normal distribution. %
    \item Experiment 4: Both the frequencies $\BFO\in\rset^{Kd}$ and the amplitudes $\boldsymbol{\hat{\beta}}\in \mathbb{C}^K$ are trained by the stochastic gradient method.
    \item Experiment 5: The $\BFO\in\rset^{Kd}$ weight distribution is obtained adaptively by Algorithm~\ref{alg:ARFM} but using the sigmoid activation function. 
\end{itemize}
For the second case we perform Experiment 1-4 and in simulations for the third case we perform Experiment 1 and Experiment 3. All chosen parameter values are presented in Table \ref{table:parameter_comparison}.

We denote by $\BARS_{\text{test}}\in\mathbb{C}^{\tilde{N}\times K}$ the matrix with elements %
$e^{{\IM}\omega_k\cdot \tilde x_n}$.  The test data $\{(\tilde x_n, \tilde y_n) \, |\, n = 1,...,\tilde{N}\}$ are i.i.d. samples from the same probability distribution and normalized by the same empirical mean and standard deviation as the training data $\{(x_n, y_n) \,|\, n=1,\ldots, N$\}.
In the computational experiments we compute the \emph{generalization error} as
\[e_K := \sqrt{\sum_{n=1}^{\tilde{N}}|(\BARS_{\mathrm{test}}\BFBH)_n-\tilde y_n|^2}\,.\]

We denote by ${\sigma}_K$ the empirical standard deviations  of the generalization error, based on $\bar M=10$ independent realizations for each fixed $K$ and let an \emph{error bar} be the closed interval 
\begin{equation*}
    [e_{K} - 2{\sigma}_K, e_{K} + 2{\sigma}_K].
\end{equation*}

The purpose of Experiment 5 is to demonstrate the possibility of changing the activation function $x\mapsto e^{{\IM}\omega\cdot x}$ to the sigmoid activation function $x \mapsto \frac{1}{1 + e^{-\omega\cdot x}}$ when running Algorithm~\ref{alg:ARFM}. With such a change of activation function the concept of sampling frequencies turns into sampling weights. In practice we add one dimension to each $x$-point to add a bias, compensating for using a real valued activation function and we set the value of the additional component to one. Moreover, we change $\BARS_{n,k} = e^{{\IM}\omega_k\cdot x_n}$ in \eqref{eq:num_disc_prob} to $\BARS_{n,k} = \frac{1}{1 + e^{-\omega_k\cdot x_n}}$.

\medskip\noindent
{\it Case 1: Target function with a regularised discontinuity.}
This case tests the capability of Algorithm~\ref{alg:ARFM} and Algorithm~\ref{alg:ARFM_acov} to approximately find and sample frequencies $\BFO$ from the optimal distribution $p_* = |\hat f|/\|\hat f\|_{L^1(\rset)}$. The target function \[f(x) = \mathrm{Si}\left(\frac{x}{a}\right)e^{-\frac{x^2}{2}}\] where $a = 10^{-3}$ and \[\mathrm{Si}(x) := \int_0^x \frac{\sin(t)}{t}\mathrm{d}t\] is the so called \emph{Sine integral} has a Fourier transform that decays slowly as $\omega^{-1}$ up to $|\omega| = 1/a = 1000$. The target function $f$ is plotted  in Figure \ref{fig:f_zoomed}, together with $f$ evaluated in $N$ $x$-points from a standard normal distribution over an interval chosen to emphasize that $N$ points is enough to resolve the local oscillations near $x = 0$.
\begin{figure}[ht]
\centering
\includegraphics[width=0.95\textwidth]{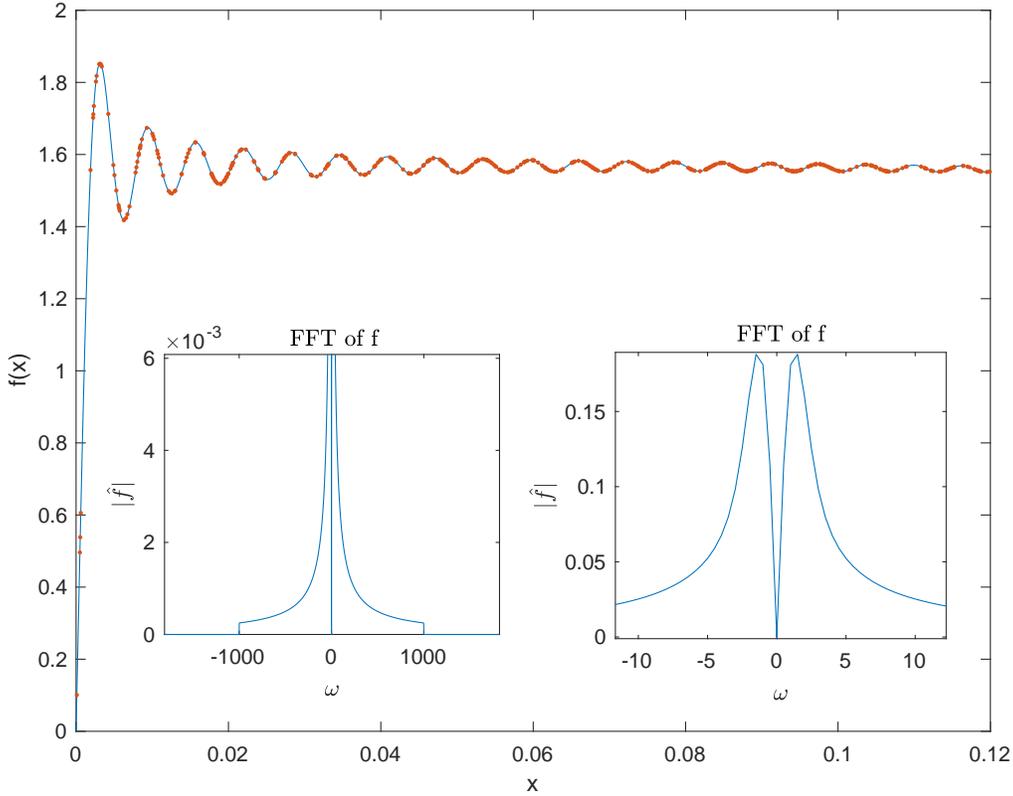}
\caption{Case 1: Graph of the target function $f$ with sampled data set points $(x_n,y_n)$ marked (red on-line). 
The inset shows  $|\hat f|$ of its Fourier transform
and the detail of its behaviour at the origin. }\label{fig:f_zoomed}

\end{figure}

The Fourier transform of $f$ is approximated by computing the fast Fourier transform of $f$ evaluated in $2N$ equidistributed $x$-points in the interval $[-2\pi, 2\pi]$. The inset in Figure~\ref{fig:f_zoomed}  presents the absolute value of the fast Fourier transform of $f$ where we can see that the frequencies drop to zero at approximately $|\BFO| = 1/a = 10^3$.

We generate training data and test data as follows. First sample $N$ $x$-points from $\mathcal{N}(0,1)$. Then evaluate the target function in each $x$-point to get the $y$-points and run Algorithm~\ref{alg:normalization_of_data} on the generated points to get the normalized training data $\{x_n, y_n\}_{n=1}^N$ and analogously the normalized test data $\{\tilde{x}_n, \tilde{y}_n\}_{n=1}$.

We run Experiment 1--5 on the generated data for different values of $K$ and present the resulting generalization error dependence on $K$, with error bars, in Figure~\ref{fig:almost_disc}. The triangles pointing to the left represent generalization errors produced from a neural network trained by Algorithm~\ref{alg:ARFM} and the diamonds by Algorithm~\ref{alg:ARFM_acov}. The stars correspond to the stochastic gradient descent with initial frequencies from 
$\mathcal{N}(0, 50^2)$ while the circles also corresponds to the stochastic gradient descent but with initial frequencies from $\mathcal{N}(0, 1)$. The squares correspond to the standard random Fourier features approach sampling frequencies $\omega_k$ from $\mathcal{N}(0, 1)$. 
The triangles pointing down represent Algorithm~\ref{alg:ARFM} but for the sigmoid activation function. Algorithm~\ref{alg:ARFM} show a constant slope with respect to $K$. Although the generalization error becomes smaller for the stochastic gradient descent as the variance for the initial frequencies increase to $50^2$ from $1$, it stagnates as $K$ increases. 
For a given $K$ one could fine tune the initial frequency distribution for the stochastic gradient descent but for Algorithm~\ref{alg:ARFM} and Algorithm~\ref{alg:ARFM_acov} no such tuning is needed. 
The specific parameter choices for each experiment are presented in Table~\ref{table:parameter_comparison}.

\begin{figure}[ht]
\begin{subfigure}{0.49\textwidth}
\centering
\includegraphics[width=0.95\linewidth]{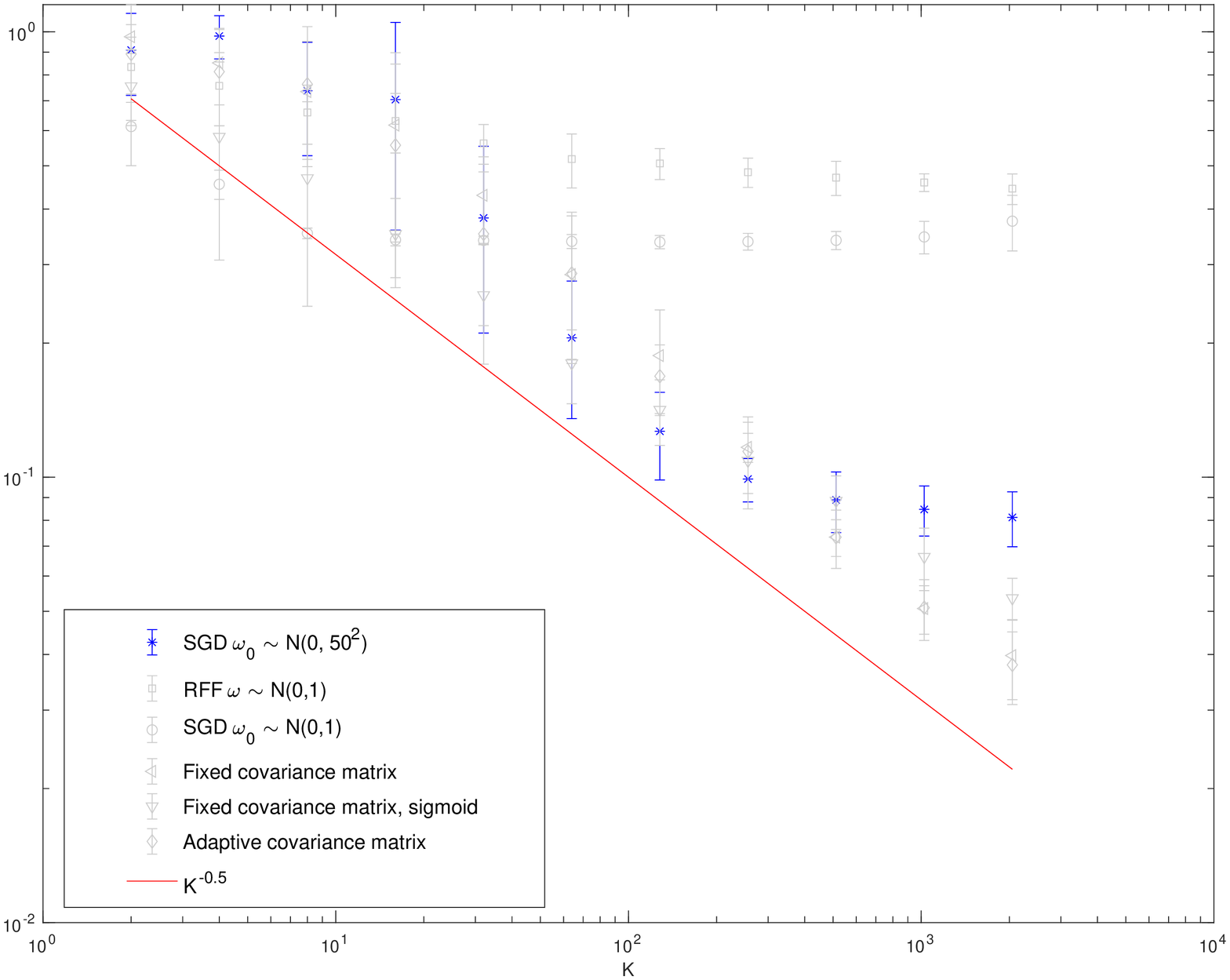}
\caption{Experiment 4, Stochastic gradient method with a large variance on initial components of $\BFO$}
\label{fig:almost_disc_t1}
\end{subfigure}
\begin{subfigure}{0.49\textwidth}
\centering
\includegraphics[width=0.95\linewidth]{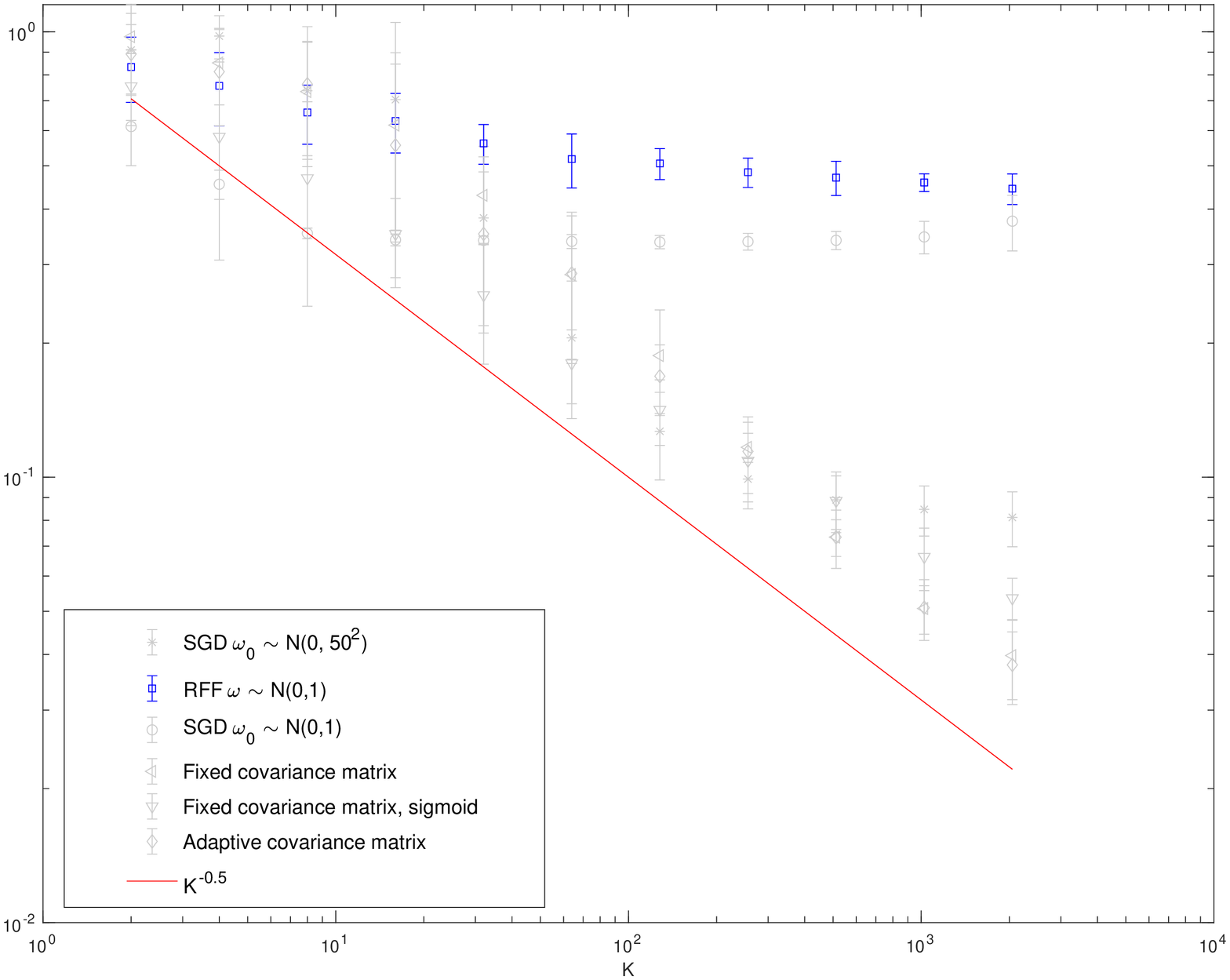}
\caption{Experiment 3, Random Fourier Features}
\label{fig:almost_disc_t2}
\end{subfigure}

\begin{subfigure}{0.49\textwidth}
\centering
\includegraphics[width=0.95\linewidth]{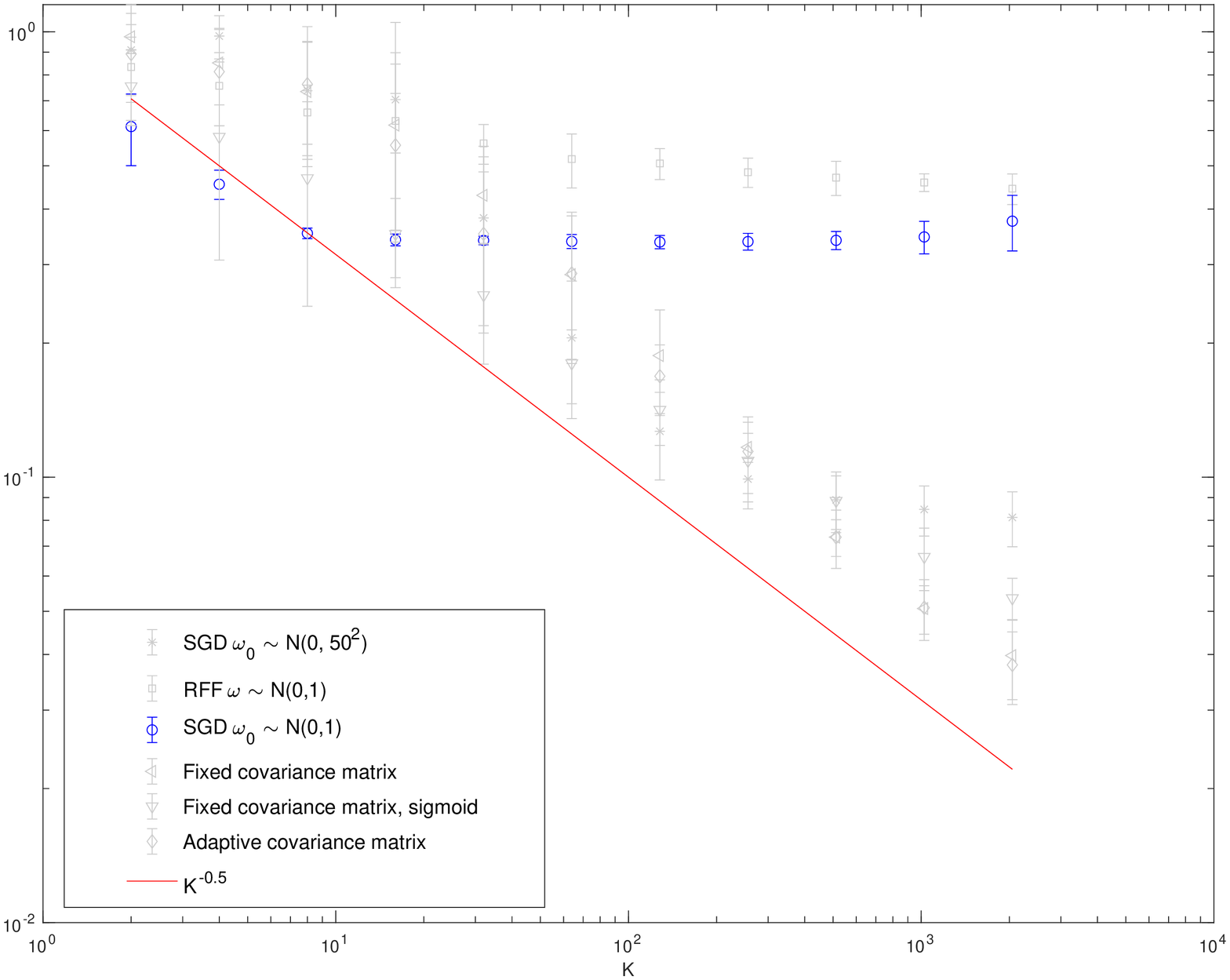}
\caption{Experiment 4, Stochastic gradient method with initial components of $\BFO$ from $\mathcal{N}(0,1)$}
\label{fig:almost_disc_t3}
\end{subfigure}
\begin{subfigure}{0.49\textwidth}
\centering
\includegraphics[width=0.95\linewidth]{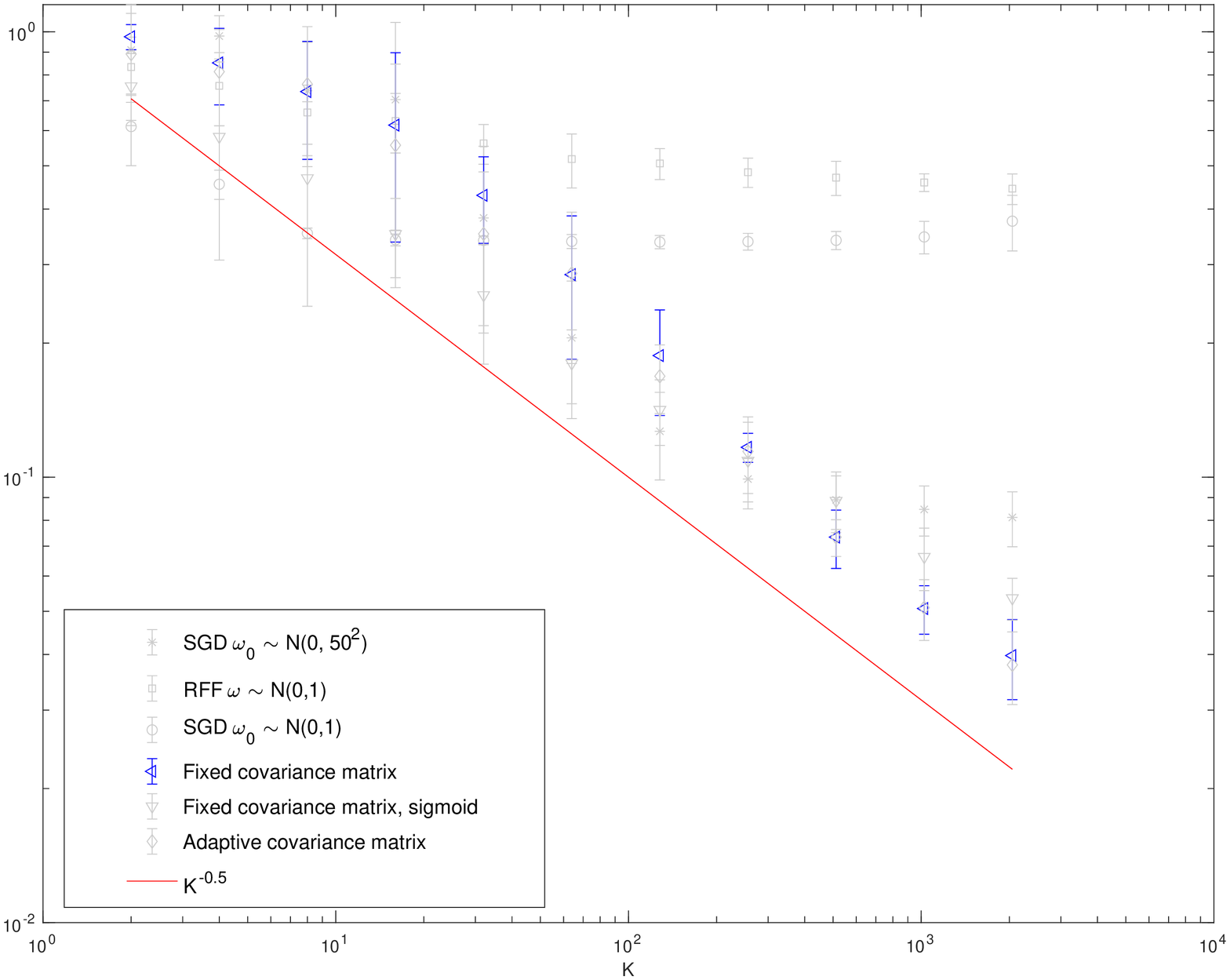}
\caption{Experiment 1, Adaptive Metropolis sampling}
\label{fig:almost_disc_t4}
\end{subfigure}

\begin{subfigure}{0.49\textwidth}
\centering
\includegraphics[width=0.95\linewidth]{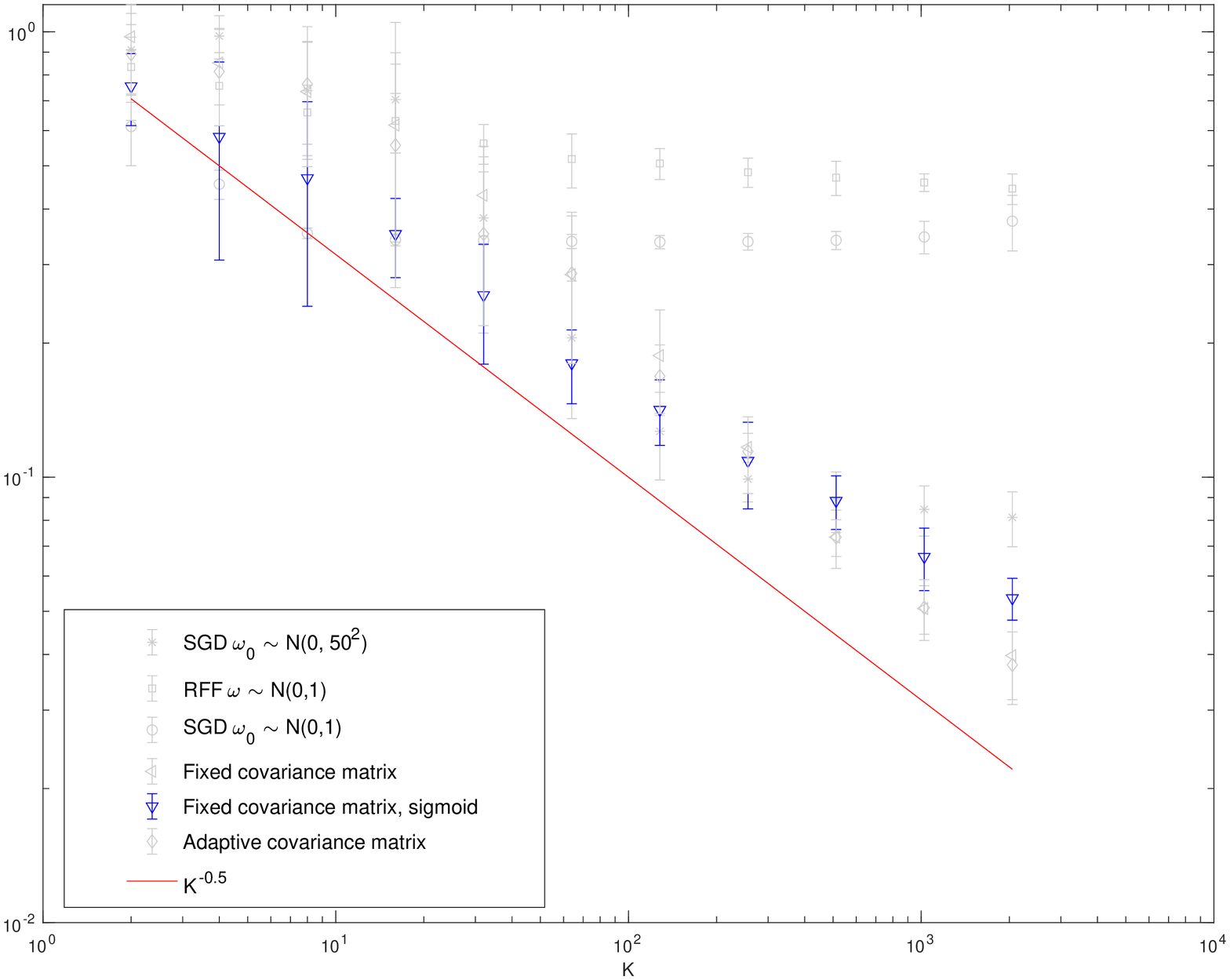}
\caption{Experiment 5, Adaptive Metropolis sampling with the sigmoid activation function}
\label{fig:almost_disc_t5}
\end{subfigure}
\begin{subfigure}{0.49\textwidth}
\centering
\includegraphics[width=0.95\linewidth]{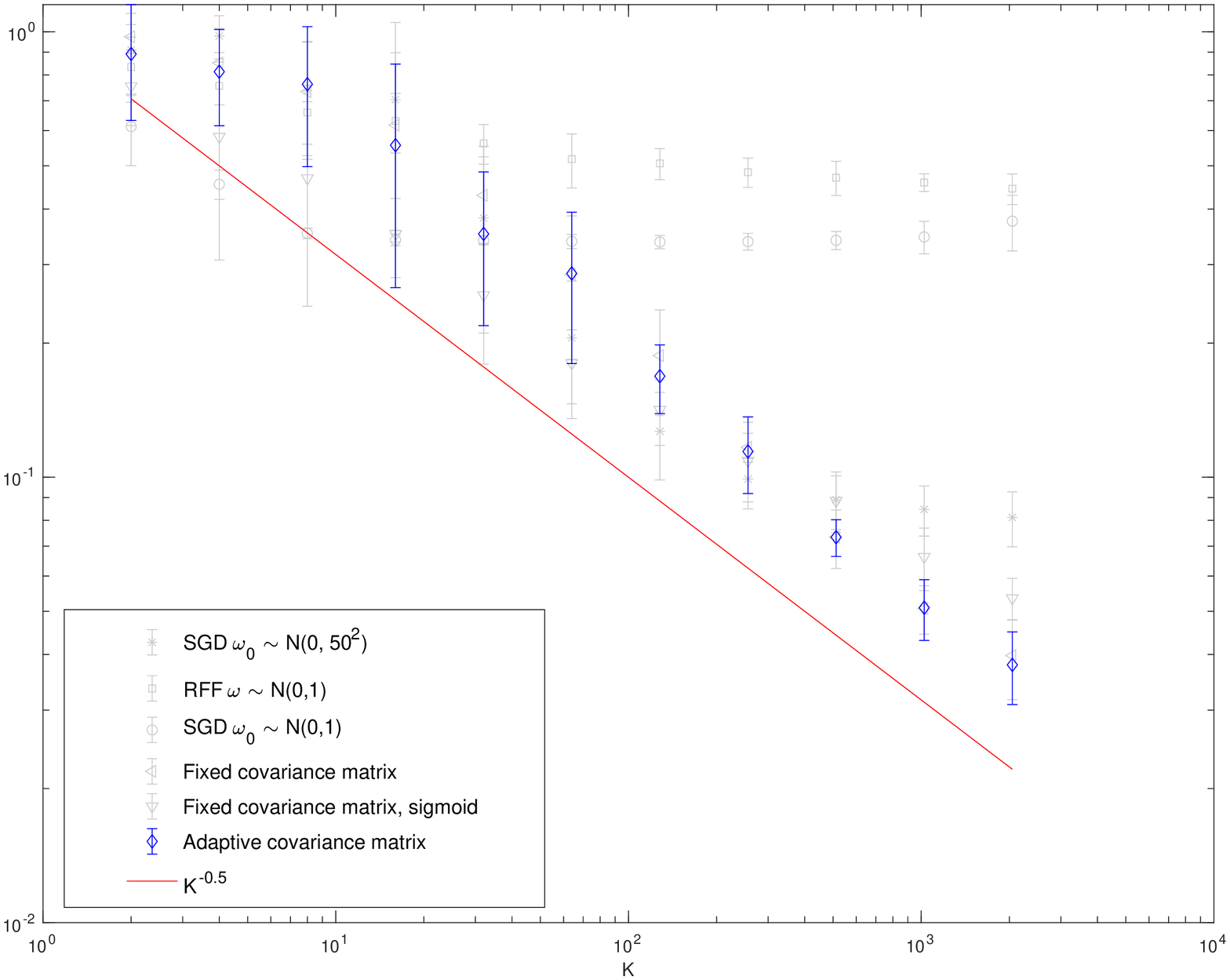}
\caption{Experiment 2, Adaptive Metropolis sampling with adaptive covariance matrix}
\label{fig:almost_disc_t6}
\end{subfigure}
\caption{Case 1: The same data are shown in all figures with each of the six different experiments highlighted (blue on-line).}\label{fig:almost_disc}
\end{figure}

\medskip\noindent
{\it Case 2: A high dimensional target function.}
The purpose of this case is to test the ability of Algorithm~\ref{alg:ARFM} to train a neural network in a higher dimension. Therefore we set $d = 5$. The data is generated analogously to how it is generated in Case 1 but we now use the target function $f:\rset^5\to \rset$
\begin{equation*}
    f(x) = \text{Si}\left(\frac{x_1}{a}\right)e^{-\frac{|x|^2}{2}}
\end{equation*}
where $a = 10^{-1}$. We run Experiment 1, 3 and 4 and the resulting convergence plot with respect to the number of frequencies $K$ is presented in Figure \ref{fig:hd_ge}.

In Figure \ref{fig:hd_ge} we can see the expected convergence rate of $\mathcal{O}(K^{-1/2})$ for Algorithm~\ref{alg:ARFM}. The Stochastic gradient method gives an error that converges as $\mathcal{O}(K^{-1/2})$ for smaller values of $K$ but stops improving for approximately $K>128$ for the chosen number of iterations. 

\begin{figure}[ht]
\centering
\includegraphics[width=0.99\textwidth]{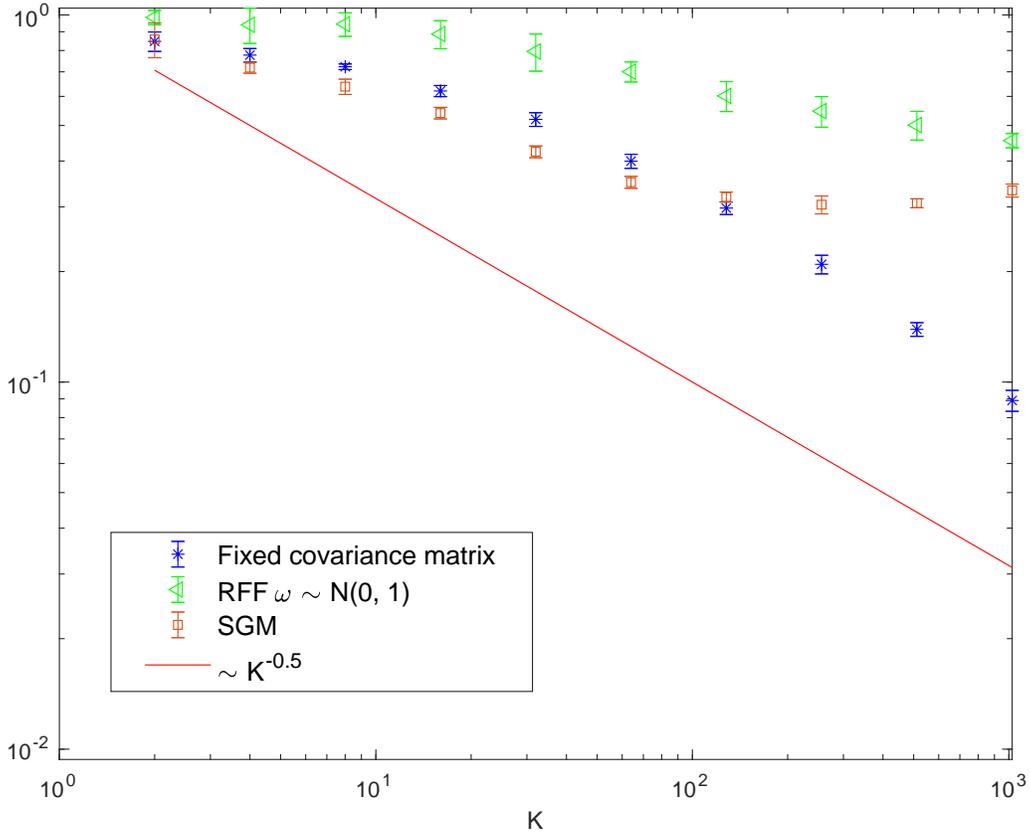}
\caption{Case 2: The figure illustrates the generalization error with respect to $K$ for a target function in dimension $d = 5$.}\label{fig:hd_ge}
\end{figure}

\medskip\noindent
{\it Case 3: Anisotropic Gaussian target function.} 
Now we consider the target function
\begin{equation}\label{eq:elliptic_target_fcn}
    f(x) = e^{-(32 x_1)^2/2}e^{-(32^{-1} x_2)^2/2}\,,
\end{equation}
which, as well as $\hat{f}$, has elliptically shaped level surfaces. To find the optimal distribution $p_*$ which is
$\mathcal{N}(\boldsymbol{0}, \mathrm{diag}([32^{-2}, 32^2]))$ thus requires to find the covariance matrix $\text{diag}([32^{-2}, 32^2])$.

The generation of data is done as in Case 1 except that the non normalized $x$-points are independent random vectors in $\mathbb{R}^2$ with independent $\mathcal{N}(0,1)$ components. We fix the number of nodes in the neural network to $K = 256$ and compute an approximate solution to the problem \eqref{eq:num_normal_eq} by running Experiment 1, 2 and 4.

Convergence of the generalization error with respect to time is presented in Figure \ref{fig:eta_comparisons} where we note that both Algorithm~\ref{alg:ARFM} and Algorithm~\ref{alg:ARFM_acov} produce faster convergence than the stochastic gradient descent. For the stochastic gradient method the learning rate has been tuned for a benefit of the convergence rate in the generalization error but the initial distribution of the components of $\BFO$ simply chosen as the standard normal distribution.

\begin{figure}[ht]
\centering
\includegraphics[width=0.99\textwidth]{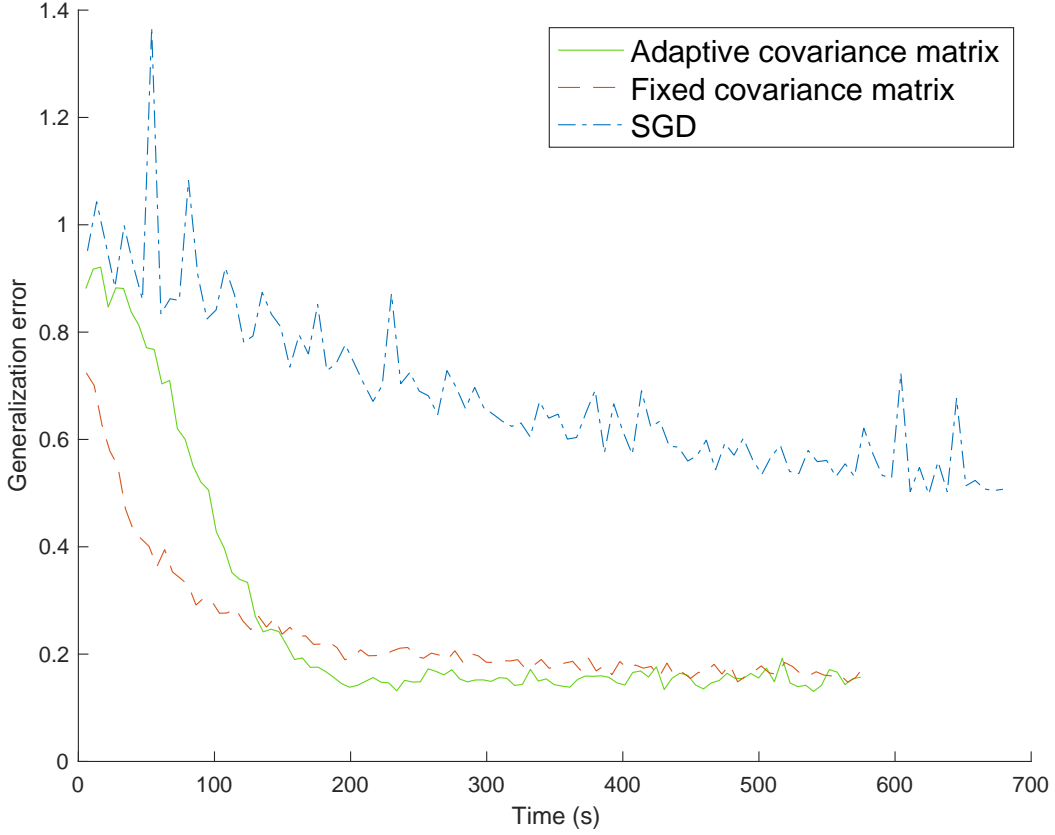}
\caption{Case 3: The figure illustrates the generalization error over time when the approximate problem solution is computed using Algorithm 1, Algorithm 2 and stochastic gradient descent.}\label{fig:eta_comparisons}
\end{figure}

\medskip\noindent
{\it Case 4: The MNIST data set.}
The previously presented numerical tests have dealt with problems of a pure regression character. In contrast, we now turn our focus to a classification problem of handwritten digits.

The MNIST data set consists of a training set of $60000$ handwritten digits with corresponding labels and a test set of $10000$ handwritten digits with corresponding labels.

We consider the ten least squares problems
\begin{equation}\label{eq:mnist_num_disc_prob}
    \min_{\boldsymbol{\hat{\beta}}^i\in \mathbb{C}^{K}}\left(N^{-1}|\BARS\boldsymbol{\hat{\beta}}^i-\mathbf y^i|^2 + \lambda|\boldsymbol{\hat{\beta}}^i|^2\right), \; i = 0,1,\dots,9\,,
\end{equation}
where $\BARS\in\mathbb{C}^{N\times K}$ is the matrix with elements $\BARS_{n,k} = e^{\mathrm{i}\omega_k\cdot x_n}$, $n = 1,...,N$, $k = 1,...,K$ and 
$\mathbf{y}^i=(y_1^i,\ldots,y_N^i)\in \mathbb{R}^N$. 
The training data $\{(x_n; (y_n^0, y_n^1,...,y_n^9))\}_{n=1}^N$ consist of handwritten digits $x_n\in \mathbb{R}^{784}$ with corresponding vector labels $(y_n^0, y_n^1,...,y_n^9)$. 
Each vector label $(y_n^0, y_n^1,...,y_n^9)$ has one component equal to one and the other components equal to zero. The index $i$ of the component $y_n^i$ that is equal to $1$ is the number that the handwritten digit $x_n$ represents. The problems \eqref{eq:mnist_num_disc_prob} 
have the corresponding linear normal equations
\begin{equation}\label{eq:num_normal_eq_mnist}
    (\BARS^T\BARS+\lambda N \mathbf{I})
    \boldsymbol{\hat{\beta}}^i = \BARS^T\mathbf y^i\,, \;\;\; i = 0,1,...,9\,,
\end{equation}
which we solve using the $\texttt{MATLAB}$
backslash operator for each $i = 0,1,...,9$.
The regularizing parameter $\lambda$ acts as a regulator to adjust the bias-variance trade-off.

We compute an approximate solution to the problem \eqref{eq:mnist_num_disc_prob} by using 
Algorithm~\ref{alg:ARFM} but in the Metropolis test step evaluate 
$||(\hat{\beta}^0_k, \hat{\beta}^1_k,...,\hat{\beta}^9_k)'||_2^\ALPHAEX/||(\hat{\beta}^0_k, \hat{\beta}^1_k,...,\hat{\beta}^9_k)||_2^\ALPHAEX>r_{\mathcal{U}}$ where $||\cdot||_{2}$ denotes the Euclidean norm $||\hat{\beta}_k||_{2} = \sqrt{\sum_{i = 0}^9 (\hat{\beta}_k^i)^2}$.

We evaluate the trained artificial neural network
for each handwritten test digit 
$\tilde{x}_n, \; n=1,2,...,\tilde{N}$ and classify the handwritten test digit as the number
\begin{equation*}
    \argmax_i \{|\sum_{k=1}^{K}\hat{\beta}_k^{i}s(\omega_k\cdot \Tilde{x}_n)|\}_{i=0}^9\,, 
\end{equation*}
where $\{\omega_k; \hat{\beta}_k^{0}, \hat{\beta}_k^{1},...,\hat{\beta}_k^{9}\}_{k=1}^K$ are the trained frequencies and amplitudes resulting from running Algorithm~\ref{alg:ARFM}.

As a comparison we also use frequencies from the standard normal distribution and from the normal distribution $\mathcal{N}(0, 0.1^2)$ but otherwise solve the problem the same way as described in this case.

The error is computed as the percentage of misclassified digits over the test set $\{(\Tilde{x}_n; (\Tilde{y}_n^0, \Tilde{y}_n^1,...,\Tilde{y}_n^9))\}_{n=1}^{\tilde{N}}$. We present the results in 
Table~\ref{table:MNIST_percent_mistaken} and 
Figure~\ref{fig:mnist_ge} where we note that the smallest error is achieved when frequencies are sampled by using Algorithm~\ref{alg:ARFM}. When sampling frequencies from the standard normal distribution, i.e, $\mathcal{N}(0, 1)$, we do not observe any convergence with respect to $K$.

\begin{table}[ht]
\begin{tabular}{l|l|l|l|l|l}
\cline{2-5}
 & K    & Fixed  & Fixed & Adaptive  &  \\ 
 &      & $\omega\sim\mathcal{N}(0, 1)$ & $\omega\sim\mathcal{N}(0, 0.1^2)$& &\\ \cline{2-5}
 & 2    & 89.97\% & 81.65\% & 80.03\%  &  \\ \cline{2-5}
 & 4    & 89.2\% & 70.65\% & 65.25\%  &  \\ \cline{2-5}
 & 8    & 88.98\% & 55.35\% & 53.44\%  &  \\ \cline{2-5}
 & 16   & 88.69\% & 46.77\% & 36.42\%  &  \\ \cline{2-5}
 & 32   & 88.9\% & 30.43\% & 23.52\%  &  \\ \cline{2-5}
 & 64   & 88.59\% & 19.73\% & 16.98\%  &  \\ \cline{2-5}
 & 128  & 88.7\% & 13.6\% & 11.13\%  &  \\ \cline{2-5}
 & 256  & 88.09\% & 10.12\% & 7.99\%  &  \\ \cline{2-5}
 & 512  & 88.01\% & 8.16\% & 5.93\%  &  \\ \cline{2-5}
 & 1024 & 87.34\% & 6.29\% & 4.57\%  &  \\ \cline{2-5}
 & 2048 & 86.5\% & 4.94\% & 3.5\%  &  \\ \cline{2-5}
 & 4096 & 85.21\% & 3.76\% & 2.74\%  &  \\ \cline{2-5}
 & 8192 & 83.98\% & 3.16\% & 1.98\%  &  \\ \cline{2-5}
\end{tabular}
\caption{Case 4: The table shows the percentage of misclassified digits 
in the MNIST test data set for different values of $K$. Comparison between adaptively computed distribution of frequencies $\omega_k$ and sampling a fixed (normal) distribution.}\label{table:MNIST_percent_mistaken}
\end{table}

\begin{figure}[ht]
\centering
\includegraphics[width=0.99\textwidth]{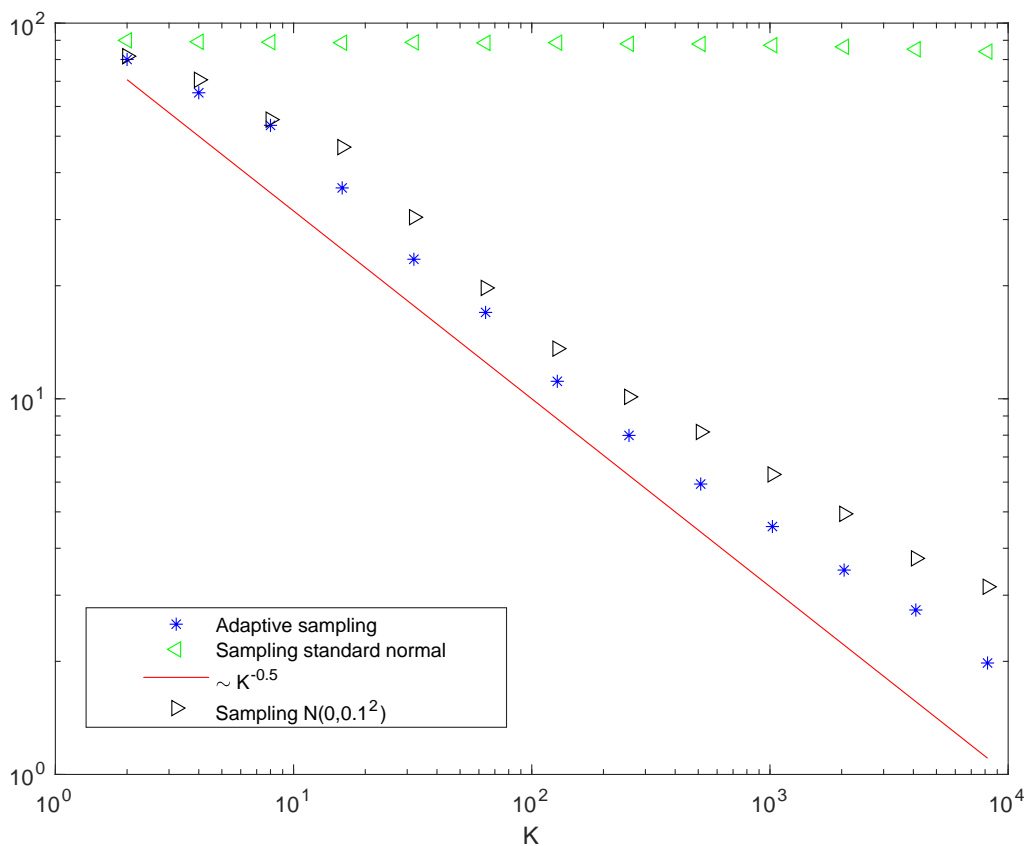}
\caption{Case 4: Dependence on $K$ of the misclassification percentage in the MNIST.}\label{fig:mnist_ge}
\end{figure}

\begin{table}[ht]%
\tiny{
\bgroup
\def\arraystretch{1.2}
\begin{tabular}{|l|l|l|l|l|l|l|l|l|l|l|l|l|}
\hline
                                                        & \multicolumn{10}{l|}{\textbf{Regression}}                                                                                                                                                                                                                                                                                                                                                                                                        & \multicolumn{2}{l|}{\textbf{Classification}}                                                                                   \\ \hline
Case                                                    & \multicolumn{5}{l|}{\begin{tabular}[c]{@{}l@{}}Case 1:\\ a regularised step function\end{tabular}}                                                                                   & \multicolumn{2}{l|}{\begin{tabular}[c]{@{}l@{}}Case 2:\\ a high dimensional\\ function\end{tabular}}                                             & \multicolumn{3}{l|}{\begin{tabular}[c]{@{}l@{}}Case 3:\\ anisotropic Gaussian\\ function\end{tabular}} & \multicolumn{2}{l|}{\begin{tabular}[c]{@{}l@{}}Case 4:\\ The MNIST\\ data set\end{tabular}}                                    \\ \hline
Purpose                                                 & \multicolumn{5}{l|}{\begin{tabular}[c]{@{}l@{}}Find $p$ such that the constant\\ $\E_\omega[\frac{|\hat f(\omega)|^2}{(2\pi)^{d}p^2(\omega)}]$\\ does not become large\end{tabular}} & \multicolumn{2}{l|}{\begin{tabular}[c]{@{}l@{}}Study the ability\\ of Alg. 1 to find a\\ high dimensional\\ dimensional\\ function\end{tabular}} & \multicolumn{3}{l|}{\begin{tabular}[c]{@{}l@{}}Computational\\ complexity comparison\end{tabular}}     & \multicolumn{2}{l|}{\begin{tabular}[c]{@{}l@{}}Study the ability\\ of Alg. 1 to work\\ with non\\ synthetic data\end{tabular}} \\ \hline
\begin{tabular}[c]{@{}l@{}}Target\\ $f(x)$\end{tabular} & \multicolumn{5}{l|}{\begin{tabular}[c]{@{}l@{}}$\text{Si}\left(\frac{x}{a}\right)e^{-\frac{x^2}{2}}$\\ $a = 10^{-3}$\end{tabular}}                                                   & \multicolumn{2}{l|}{\begin{tabular}[c]{@{}l@{}}$\text{Si}\left(\frac{x_1}{a}\right)e^{-\frac{|x|^2}{2}}$\\ $a = 10^{-1}$\end{tabular}}           & \multicolumn{3}{l|}{$e^{-(32 x_1)^2/2}e^{-(32^{-1} x_2)^2/2}$}                                         & \multicolumn{2}{l|}{}                                                                                                          \\ \hline
$d$                                                     & \multicolumn{5}{l|}{$1$}                                                                                                                                                             & \multicolumn{2}{l|}{$5$}                                                                                                                         & \multicolumn{3}{l|}{$2$}                                                                               & \multicolumn{2}{l|}{$784$}                                                                                                     \\ \hline
$K$                                                     & \multicolumn{5}{l|}{$2^i, \, i = 1,2,...,11$}                                                                                                                                        & \multicolumn{2}{l|}{$2^i, \, i = 1,2,...,10$}                                                                                                    & \multicolumn{3}{l|}{$256$}                                                                             & \multicolumn{2}{l|}{$2^i, \, i = 1,2,...,13$}                                                                                  \\ \hline
Experiment                                              & Exp. 1                  & Exp. 2                 & Exp. 3                  & Exp. 4                          & Exp. 5                                                                & Exp. 1                                                                   & Exp. 4                                                                & Exp. 1                       & Exp. 2                         & Exp. 4                                 & Exp. 1                                                         & Exp. 3                                                        \\ \hline
Method                                                  & Alg. 1                  & Alg. 2                 & RFF$^{1}$               & SGM                             & \begin{tabular}[c]{@{}l@{}}Alg. 1\\ sigmoid\end{tabular}              & Alg. 1                                                                   & SGM                                                                   & Alg. 1                       & Alg. 2                         & SGM                                    & Alg. 1                                                         & RFF$^{2}$                                                     \\ \hline
$N$                                                     & $10^4$                  & $10^4$                 & $10^4$                  & $10^4$                          & $10^4$                                                                & $10^4$                                                                   & $10^4$                                                                & $10^4$                       & $10^4$                         & $3\times 10^7$                         & $6\times 10^4$                                                 & $6 \times 10^4$                                               \\ \hline
$\tilde{N}$                                             & $10^4$                  & $10^4$                 & $10^4$                  & $10^4$                          & $10^4$                                                                & $10^4$                                                                   & $10^4$                                                                & $10^4$                       & $10^4$                         & $10^4$                                 & $10^4$                                                         & $10^4$                                                        \\ \hline
$\ALPHAEX$                                              & $3d-2$                  & $3d-2$                 &                         &                                 & $3d-2$                                                                & $3d-2$                                                                   &                                                                       & $3d-2$                       & $3d-2$                         &                                        & $3d-2$                                                         &                                                               \\ \hline
$\lambda$                                               & $0.1$                   & $0.1$                  & $0.1$                   & $0$                             & $0.1$                                                                 & $0.1$                                                                    & 0                                                                     & $0.1$                        & $0.1$                          & $0$                                    & $0.1$                                                          & $0.1$                                                         \\ \hline
$M$                                                     & $10^3$                  & $5000$                 & N/A                     & $10^7$                          & $10^4$                                                                & $2.5\times 10^3$                                                         & $10^7$                                                                & $10^4$                       & $10^4$                         & $3\times 10^7$                         & $10^2$                                                         &                                                               \\ \hline
$\bar{M}$                                               & $10$                    & $10$                   & $10$                    & $10$                            & $10$                                                                  & $10$                                                                     & $10$                                                                  & $1$                          & $1$                            & $1$                                    & $1$                                                            & $1$                                                           \\ \hline
$\delta$                                                & $2.4^2/d$               & $0.1$                  &                         &                                 & $2.4^2/d$                                                             & $\frac{2.4^2}{10d}$                                                      &                                                                       & $0.5$                        & $0.1$                          &                                        & $0.1$                                                          &                                                               \\ \hline
$\Delta t$                                              &                         &                        &                         & $\mathtt{1.5e-4}$               &                                                                       &                                                                          & $\mathtt{3.0e-4}$                                                     &                              &                                & $\mathtt{1.5e-3}$                      &                                                                &                                                               \\ \hline
$t_0$                                                   &                         & $M/10$                 &                         &                                 &                                                                       &                                                                          &                                                                       &                              & $M/10$                         &                                        &                                                                &                                                               \\ \hline
$\BFO_{\text{max}}$                                     &                         & $\infty$               &                         &                                 &                                                                       &                                                                          &                                                                       &                              & $\infty$                       &                                        &                                                                &                                                               \\ \hline
$m$                                                     & $10$                    & $50$                   &                         &                                 & $100$                                                                 & $25$                                                                     &                                                                       & $100$                        & $100$                          &                                        & $M+1$                                                          &                                                               \\ \hline
\end{tabular}
\egroup
\caption{Summary of numerical experiments together with their corresponding parameter choices. \\
$^1$ -- Random Fourier features with frequencies sampled from the fixed distribution $\mathcal{N}(0,1)$\\
$^2$ -- Random Fourier features with frequencies sampled from the fixed distribution $\mathcal{N}(0,1)$, or 
$\mathcal{N}(0,0.1^2)$ }
\label{table:parameter_comparison}
}
\end{table}

\subsection{Optimal distribution $p_*$} 
In this numerical test we demonstrate how the average generalization error depends on the distribution $p$ for a particular choice of data distribution.
We recall the frequencies $\BFO = (\omega_1, \ldots, \omega_K)$ with $(\omega_k)$ i.i.d. from the distribution $p$. %
In the experiment the data ${(x_n, y_n)}_{n=1}^N$ are given by $y_n = e^{-|x_n|^2/2} + \epsilon_n$. We let the components $\omega_k$ be sampled independently from $\mathcal{N}(0,\sigma_{\omega}^2)$ and monitor the results for different values of $\sigma_{\omega}$. Note that Algorithm~\ref{alg:ARFM} would approximately sample $\BFO$ from the optimal density $p_*$, which in this case correspond to the standard normal, i.e., $\omega_k\sim\mathcal{N}(0,1)$, see \eqref{eq:optimal_density} .

In the simulations we choose $x_n$ from $\mathcal{N}(0,1)$, $\epsilon_n$ from $\mathcal{N}\big(0,0.1^2\big)$, $d = 7$, $K = 500$, $N = 10^5$ and $\lambda = 0.01$. The error bars are estimated by generating $\bar M=10$ independent realizations for each choice of $\sigma_{\omega}$. 

The results are depicted
in Figure~\ref{fig:mean_sqr_vs_omega_vol} where we observe that the generalization error is minimized for 
$\sigma_{\omega} \approx 1$ which is in an agreement 
with the theoretical optimum.

\begin{figure}[ht]
\centering
\includegraphics[width=0.99\textwidth]{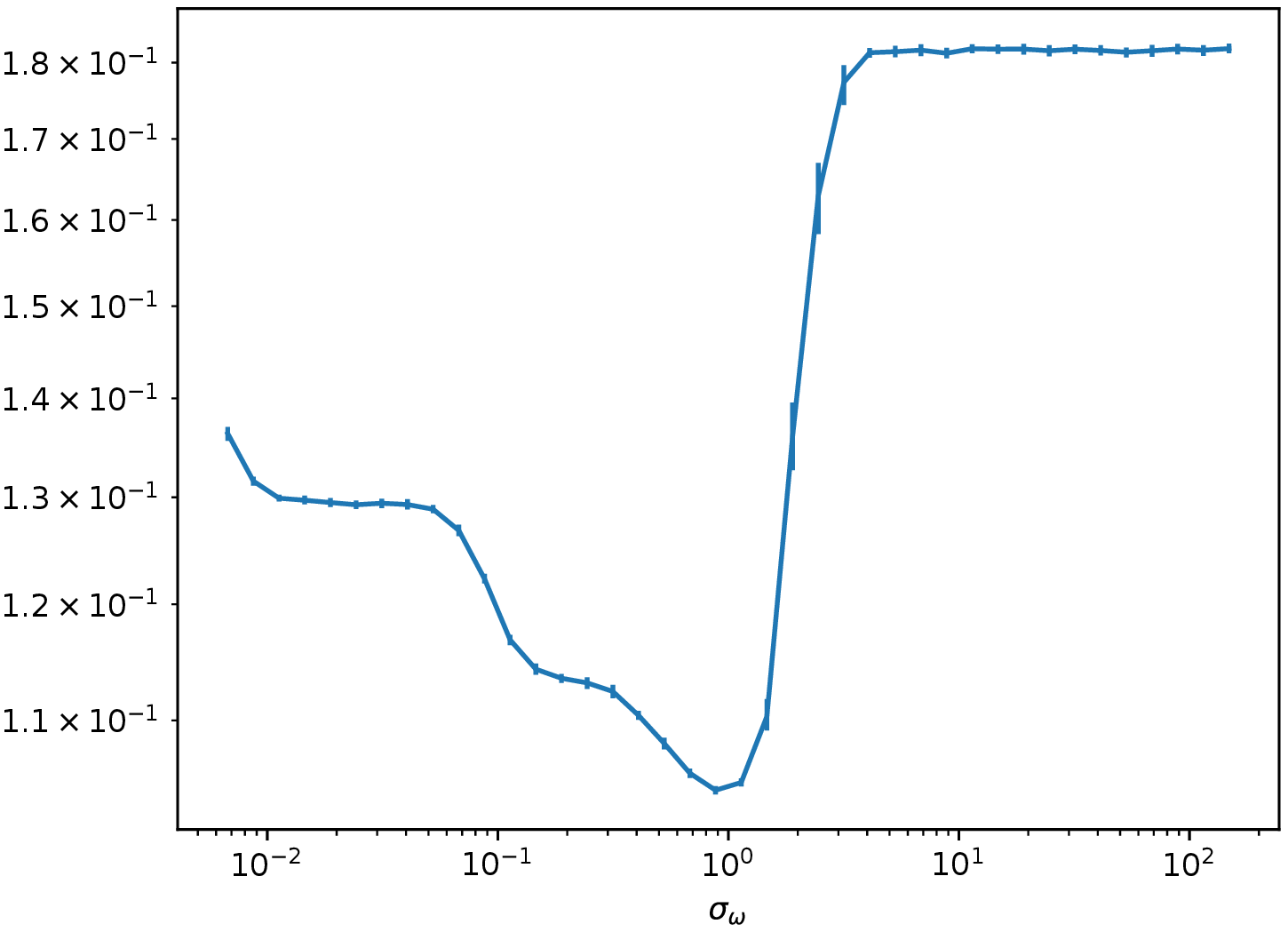}
\caption{The generalization error as a function of the standard deviation $\sigma_\omega$ of $p({\omega})$.}\label{fig:mean_sqr_vs_omega_vol}
\end{figure}

\subsection{Computing infrastructure}
The numerical experiments are computed on a desktop with an \texttt{Intel Core i9-9900K CPU @ 3.60GHz} and \texttt{32 GiB} of memory running \texttt{Matlab R2019a} under
\texttt{Windows 10 Home}.

\bibliography{references}
\bibliographystyle{plain}

\end{document}